\documentclass[11pt]{amsart}
\usepackage[pagewise]{lineno}

\usepackage{hyperref} 
\usepackage{amsmath} 
\usepackage[dvips]{graphicx}
\usepackage{amsfonts}
\usepackage[T1]{fontenc}
\usepackage{mathptmx}
\usepackage{color}
\usepackage[titletoc]{appendix}

\DeclareMathOperator{\Ad}{Ad}

\DeclareMathOperator{\rank}{rank}

\newcommand{\bbar}{\begin{pmatrix}}
\newcommand{\ebar}{\end{pmatrix}}

\newcommand{\bdm}{\begin{displaymath}}
\newcommand{\edm}{\end{displaymath}}
\newcommand{\beq}{\begin{equation}}
\newcommand{\beqa}{\begin{eqnarray}}
\newcommand{\beqas}{\begin{eqnarray*}}
\newcommand{\eeq}{\end{equation}}
\newcommand{\eeqa}{\end{eqnarray}}
\newcommand{\eeqas}{\end{eqnarray*}}
\newcommand{\dd}{\textup{d}}

\newcommand{\C}{{\mathbb C}}
\newcommand{\D}{{\mathbb D}}

\newcommand{\real}{{\mathbb R}}
\newcommand{\SSS}{{\mathbb S}}

\newcommand{\bt}{{\bf t}}
\newcommand{\bn}{{\bf n}}
\newcommand{\bb}{{\bf b}}

\newcommand{\sign}{\mathrm{sign}}

   \newtheorem{theorem}{Theorem}[section]
   \newtheorem{proposition}[theorem]{Proposition}

   \newtheorem{lemma}[theorem]{Lemma}

 \theoremstyle{remark}
   
   \newtheorem{remark}[theorem]{Remark}
   
\numberwithin{equation}{section}
\numberwithin{theorem}{section}

\begin{document}

\title{Families of spherical surfaces and harmonic maps}
\author{David Brander}
\address{Department of Applied Mathematics and Computer Science\\ Matematiktorvet, Building 303 B\\
Technical University of Denmark\\
DK-2800 Kgs. Lyngby\\ Denmark}
\email{dbra@dtu.dk}

\author{Farid Tari}
\address{Instituto de Ci\^encias Matem\'aticas e de Computa\c{c}\~ao - USP\\
Avenida Trabalhador S\~ao-Carlense, 400 - Centro \\
CEP: 13566-590 - S\~ao Carlos - SP, Brazil.\\}
\email{ faridtari@icmc.usp.br}


\begin{abstract}
We study singularities of constant positive Gaussian curvature surfaces and determine the way they bifurcate in generic 1-parameter families of such surfaces. We construct the bifurcations explicitly
using loop group methods. Constant Gaussian curvature surfaces correspond to harmonic maps, and
we examine the relationship between the two types of maps and their singularities.  Finally, we determine
 which finitely $\mathcal A$-determined map-germs from the plane to the plane can be represented by harmonic maps.
\end{abstract}
\maketitle

\section{Introduction}
Constant positive Gaussian curvature surfaces, called {\it spherical surfaces}, are related to harmonic maps $N:\Omega\to \SSS^2$, from a domain $\Omega\subset \mathbb R^2\sim \mathbb C$ to the unit sphere $\SSS^2\subset \mathbb R^3$. A spherical surface can also be realized as a parallel of a constant mean curvature (CMC) surface. 
Parallels are wave fronts and parallels of general surfaces are well studied (see for example 
\cite{ArnoldWavefront,arnold1990,bruceParallel}).

There are no \emph{complete} spherical surfaces other than the round sphere.
However, there is a  rich global class of spherical surfaces defined in terms of harmonic maps
(see Section \ref{sec:Prelm}), the global study of which necessitates the
introduction of surfaces with singularities
(Figure \ref{fig:gen_examples}).
In \cite{spherical}, a study of these surfaces from this point of view was carried out, with the
goal of getting a sense of what spherical surfaces typically look like in the large.  
Visually, singularities are perhaps the most obvious landmarks on a surface, and therefore
an essential task is to determine the generic (or stable) singularities of a surface class.
After the stable singularities, the next most common singularity type are the 
\emph{bifurcations} in generic 1-parameter families of the surfaces.  Understanding these
for spherical surfaces is the motivation for this work.  

\begin{figure}[htb]
	\centering
	$
	\begin{array}{ccc}
\includegraphics[height=32mm]{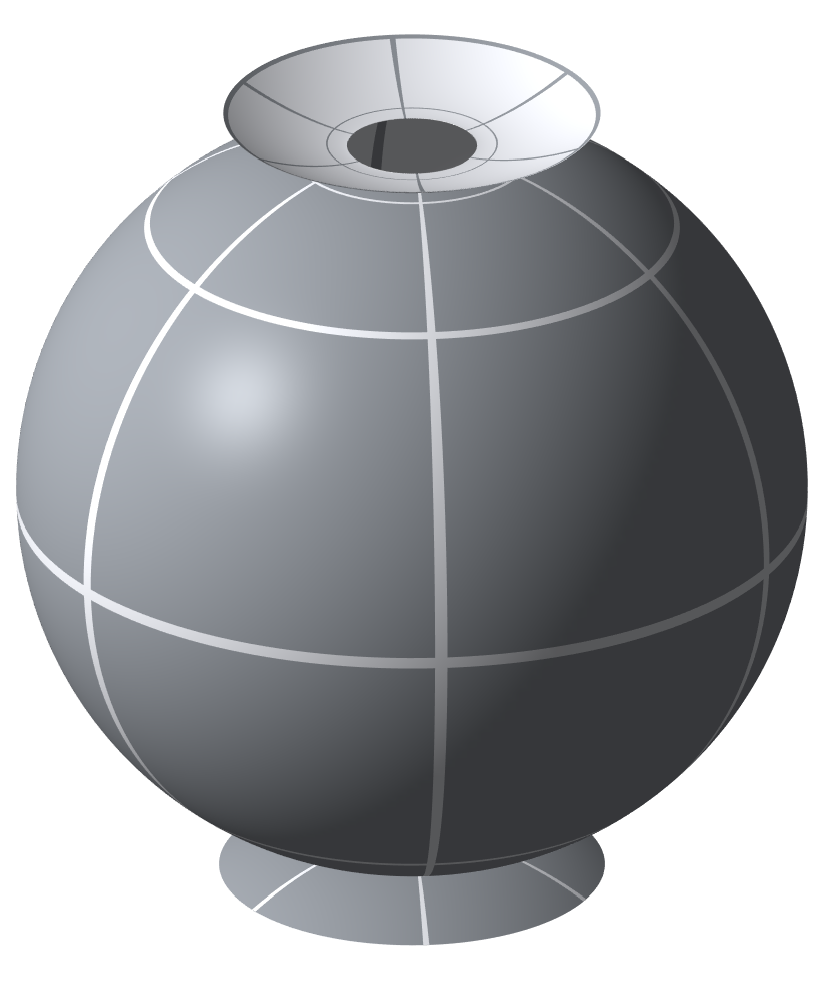}  \quad & 
\includegraphics[height=32mm]{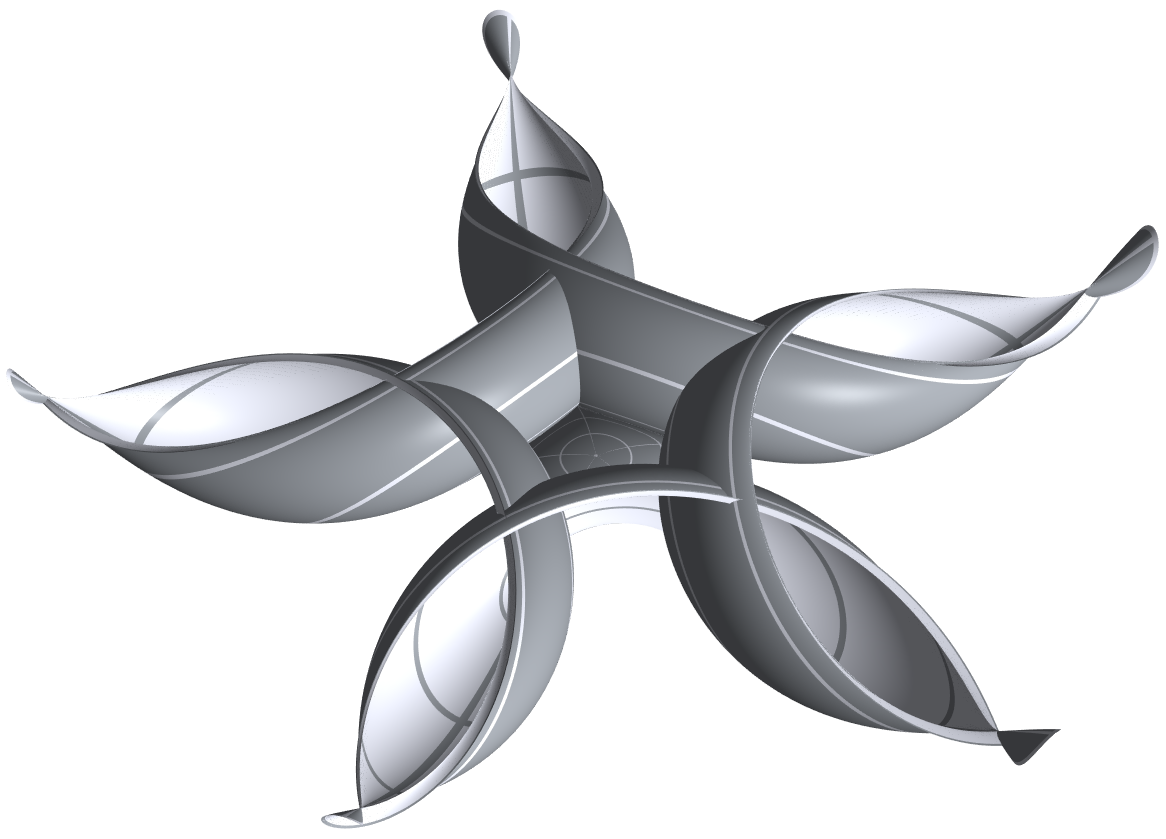}   & \quad
	\includegraphics[height=32mm]{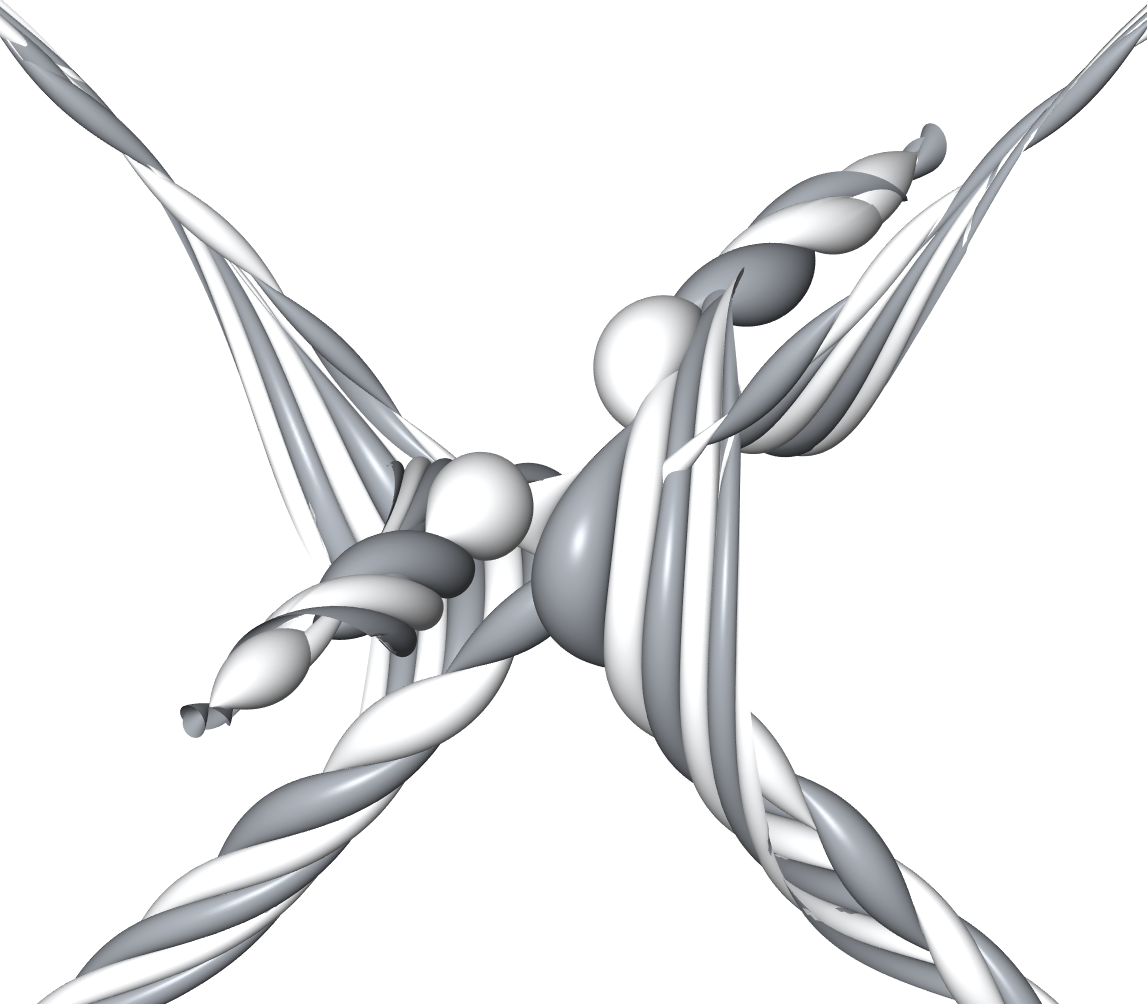}  
	\end{array}
	$
	\caption{Examples of spherical surfaces.}
	\label{fig:gen_examples}
\end{figure}

It was shown in \cite{ishimach} that the stable singularities of spherical surfaces are cuspidal edges and swallowtails (see Figure \ref{fig:StableSS}). 
It is suggested in 
\cite{spherical} that, in generic 1-parameter families of spherical surfaces, we could obtain the cuspidal beaks and
 the cuspidal butterfly bifurcations. 
We prove in this paper that indeed these are the only generic bifurcations that can occur in generic 1-parameter families
of spherical surfaces (\S \ref{sec:constructionSS}).

\begin{figure}[htp]
\begin{center}
\includegraphics[height=2.2cm]{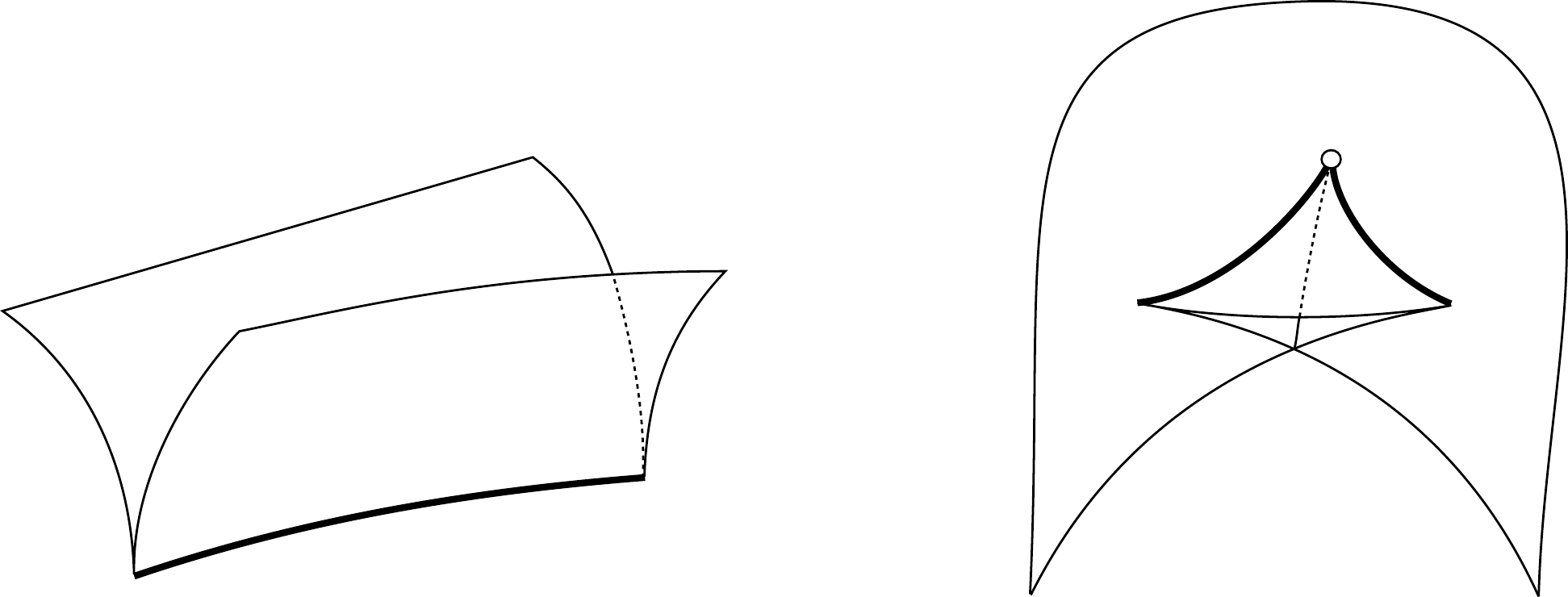}
\caption{Stable wave fronts and parallels: cuspidal edge (left) and swallowtail (right).
	 Both cases occur on spherical surfaces (\cite{ishimach}).}
\label{fig:StableSS}
\end{center}
\end{figure}

For the study of bifurcations, we use the fact that 
spherical surfaces are parallels of CMC surfaces, hence wave fronts. 
Recall that the evolutions in wave fronts are studied by Arnold in \cite{ArnoldWavefront}.
Bruce showed in \cite{bruceParallel} which of the possibilities in \cite{ArnoldWavefront} can actually occur
and proved that the generic bifurcations 
for parallels of surfaces in $\mathbb R^3$ are the following: (non-transverse) $A_3^{\pm}$, $A_4$ and $D_4^{\pm}$; see \S \ref{sec:singSephSurf} for notation, and 
Figure~\ref{fig:familiesSS}.

\begin{figure}[htp]
\begin{center}
\includegraphics[height=6.5cm]{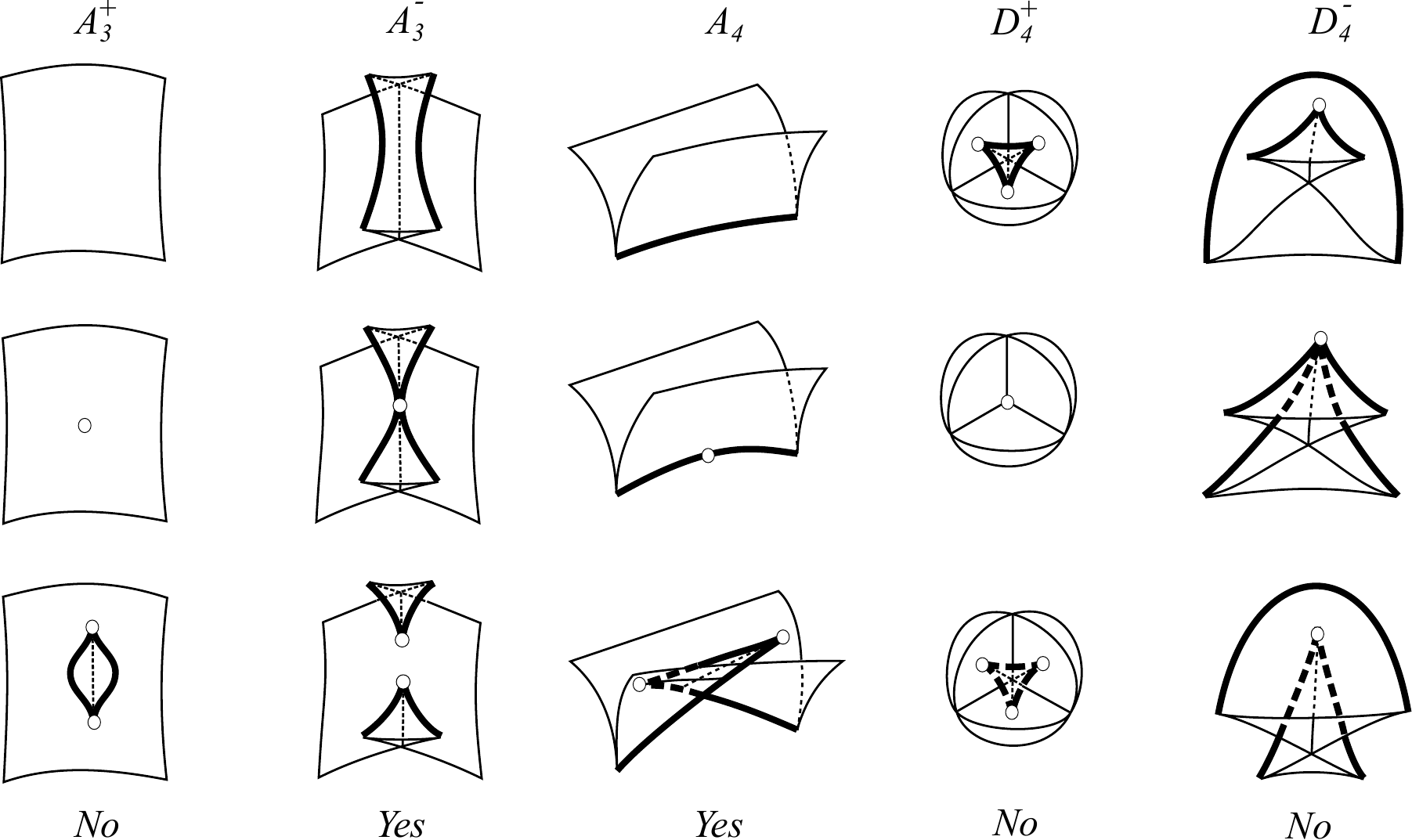}
\caption{Generic evolution of wave fronts from (\cite{ArnoldWavefront}). "Yes" for those that can occur on families of spherical surfaces and "No" for those that do not.}
\label{fig:familiesSS}
\end{center}
\end{figure}

Spherical surfaces are special surfaces, so we should not expect all the cases considered by Arnold and Bruce to occur. 
Indeed, the $A_3^{+}$ and $D_4^{\pm}$ cases do not occur for spherical surfaces (Theorem \ref{theo:PossibCodm1SingSS}).
Using loop group methods (see \S \ref{sec:DPW}), one can construct a spherical surface 
with an $A_3^{-}$ ({\it cuspidal beaks}) or an $A_4$-singularity ({\it cuspidal butterfly}), 
so these singularities do indeed occur on spherical surfaces (Theorem \ref{theo:PossibCodm1SingSS}).
The next question is whether the generic evolution of parallels of such  singularities can actually be realized by families of spherical surfaces. This is not automatic, see Remark \ref{rem:bifparalSph}. 
Using geometric criteria for $\mathcal R$-versality of families of functions established in \S \ref{sec:GeomCritA4} 
and the method described in \S \ref{sec:DPW}, we show that 
the $A_4$ and $A_3^{-}$ bifurcations do indeed occur in families of spherical surfaces
 (Theorem \ref{theo:evolA4} for  $A_4$-singularity and Theorem \ref{theo:evolA3-} for the 
non-transverse $A_3^{-}$-singularity).

In \S \ref{sec:constructionSS} we describe how to obtain examples of spherical surfaces 
exhibiting stable singularities and those that appear in generic families from geometric data along
a space curve. The proof of the existence of solutions and how to compute them using loop group methods
is given in \cite{spherical}.

In \S \ref{sec:SingHarmMap} we turn to the singularities of harmonic maps.
Wood \cite{wood1977} characterized geometrically the singularities of such maps. 
Let $\mathcal A$ be the Mather right-left group of pairs of germs of diffeomorphisms in the source and target.
We consider the problem of realization of finitely $\mathcal A$-determined singularities of map-germs from the plane to the plane and of their $\mathcal A_e$-versal deformations by germs of harmonic maps from the Euclidean plane to the Euclidean plane.
We settle this question for the $\rank 1$ finitely $\mathcal A$-determined singularities listed in \cite{rieger}  and for the simple $\rank 0$ 
singularities given in \cite{riegerruas}. 
 We show, for instance, that 
some singularities of harmonic maps can never be 
$\mathcal A_e$-versally unfolded by families of harmonic maps (Proposition \ref{prop:3jet(x,0)} and Remark \ref{rem:AboutSingHarmMap}). 
It is worth observing that the
singularities of map-germs from the plane to the plane arise as the singularities of projections of surfaces to planes or, more generally, of projections of complete intersections to planes. These projections are extensively studied; see, for example,  \cite{Arnold79,Gaffney,Goryunov,KoenVanDo,Platonova,rieger,Whitney}. We chose the list in \cite{rieger} as it exhibits all the  rank 1  germs of $\mathcal A$-codimension $\le 6$ and includes all the simple germs obtained in \cite{Goryunov}.


\section{Preliminaries}\label{sec:Prelm}
Let $\Omega$ be a simply connected open subset of $\C$, with holomorphic coordinates $z=x+iy$.
A smooth map $N: \Omega \to \SSS^2$ is harmonic if and only $N \times (N_{xx} + N_{yy} ) = 0$, i.e., 
\[
N \times N_{z \bar z} = 0.
\]
This condition is also the integrability condition for the equation
\beq \label{sphericalsurface}
f_z = i \, N \times N_z, \quad \quad \hbox {i.e.,} \quad f_x = N \times N_y, \quad f_y = - N \times N_x.
\eeq
 That is, $(f_z)_{\bar z}= (f_{\bar z})_z$ if and only if $N \times N_{z \bar z} = 0$.
Hence, given a harmonic map $N$, we can integrate the equation \eqref{sphericalsurface} to obtain a 
smooth map $f: \Omega \to \real^3$, unique up to a translation.  

A differentiable map $h: M\to \real^3$ from a surface into Euclidean space is called a \emph{frontal}
if there is a differentiable map $\mathcal{N}:M \to \SSS^2 \subset \real^3$ such that $\dd h$ is orthogonal to
 $\mathcal{N}$.  The map $h$ is called a \emph{wave front} (or \emph{front}) if
the Legendrian lift $(h, \mathcal{N}): M \to \real^3 \times \SSS^2$ is an immersion. 
The map $f$ defined above with Legendrian lift $L:=(f,N)$ is an example of a frontal.
From \eqref{sphericalsurface} the regularity of $f$ is equivalent to the regularity of $N$. At regular points
the first and second fundamental forms for $f$ are
\beqas
{I} &=& |N \times N_y|^2 \, \dd x^2 + 2\langle N \times N_y, - N \times N_x \rangle  \, \dd x \dd y
   + |N \times N_x|^2 \, \dd y^2, \\
{II}&=& \langle N , N_x \times N_y \rangle \, (\dd x^2 + \dd y^2).
	\eeqas
Thus the metric induced by the second fundamental form is conformal with respect to the conformal structure on
$\Omega$, and  the Gauss curvature of $f$ is a constant equal to $1$. 
Conversely, one can show that all regular spherical surfaces are obtained this way.
We call $f$ the \emph{spherical frontal} associated to 
the harmonic map $N$. 

 For spherical frontals we have the following characterization of the wave front condition:
\begin{proposition}\label{prop:SSfrontfrontal}
The map $f$ is a wave front near a point $p$ if and only if 
$\rank(\dd N)_p\ne 0$.
\end{proposition}

\begin{proof}
If $\rank(\dd N)_p =2$ then clearly $L=(f,N)$ is an immersion at $p$ so 
$f$ is a wave front at $p$.  
	
Suppose that $\rank(\dd N)_p =1$. We can write
$N_y= a N_x$ or $N_x = a N_y$ for some real scalar $a$.  In the first case,
from \eqref{sphericalsurface}  we have
$f_x = a N \times N_x = -a f_y$, so 
\[
\dd L = ( \dd f, \dd N) = (-a f_y, N_x) \dd x + (f_y, a N_x) \dd y,
\]
and this map has rank $2$ 
as $(-a f_y,  N_x) $ and $(f_y, a N_x)$ are not proportional.
Similarly, $N_x = aN_y$ also leads to $L$ having rank $2$. 
Therefore, $f$ is a wave front at $p$.

If $\rank(\dd N)_p =0$, then $\rank(\dd f)_p =0$ so 
$(f,N)$ is not an immersion at $p$. Consequently, $f$ is not a 
wave front at $p$ (it is only a frontal at $p$).
\end{proof}

If fact, when $\rank (\dd N )_p\ne 0$ we can say more about $f$.

\begin{proposition}\label{prop:SSasparallel}
Suppose that $\rank (\dd N )_p\ne 0$. Then the spherical surface $f$ is locally a parallel of a constant mean curvature surface.
\end{proposition}

\begin{proof}
At least one of the maps $g=f \pm N$ parameterizes locally  a
smooth and regular surface in $\mathbb R^3$. 
Indeed, consider $g=f+ \delta N$, where $\delta = \pm 1$. 
Then
$g_x=N\times N_y+\delta N_x$ and $g_y=-N\times N_x+\delta N_y$, so
$$
g_x\times g_y=-\left(\delta (|N_x|^2+  |N_y|^2)+2\varepsilon |N_x| |N_y|\sin \theta\right)N
$$
where $0\le \theta\le \pi$ is the angle between the vectors $N_x$ and $N_y$
and $\varepsilon = \sign(\det(N_y,N_x,N))$. It follows that  
$|g_x\times g_y|\ne 0$ for at least one of $\delta = \pm 1$.
The regular surface parametrized by $g$ has constant mean curvature as its
  Gauss map $N$ is harmonic, and the result follows.
\end{proof}

\begin{remark} \label{rem:bifparalSph} 
The parallels of a surface $g$ with Gauss map $N$ are given by 
$g+rN$, $r\in \mathbb R$. 
If $g$ has constant mean curvature $H_0$, then 
the parallel $g+r_0N$, with $r_0=1/(2H_0)$, has constant Gauss curvature $K=4H_0^2$ 
(see, e.g., \cite{docarmo1}; 
we took $K=1$ in the proof of Proposition \ref{prop:SSasparallel}.)
Therefore, a spherical surface is a specific parallel of a CMC surface.
This means, in particular, that 
codimension 1 phenomena that appear in the parallels of a generic surface in $\mathbb R^3$ 
by varying $r$ do not appear generically for a single spherical surface. For them to possibly occur, one has to consider 
1-parameter families of spherical surfaces.
\end{remark}

\subsection{The generalized Weierstrass representation for spherical surfaces (DPW)}\label{sec:DPW}
The method of Dorfmeister, Pedit and Wu (DPW) \cite{DorPW} gives a representation
of harmonic maps into symmetric spaces in terms of essentially arbitrary 
holomorphic functions via  a loop group decomposition.
We refer the reader to \cite{spherical} for a description of the method as it applies to
spherical surfaces. In brief,  a \emph{holomorphic potential} on an open set $U \subset \C$,
is a $1$-form
\[
\omega = \sum_{n=0}^\infty \bbar a_n (z) \lambda^{2n} & b_n (z) \lambda^{2n-1} \\
                  c_n (z)\lambda^{2n-1} & -a_n(z) \lambda^{2n} \ebar \dd z
                   = A(z) \dd z,
\]
where all component functions are holomorphic in $z$ on $U$, and with a suitable convergence
condition with respect to the auxiliary complex loop parameter $\lambda$. Such a potential can be used to produce a 
harmonic map  $N: U \to \SSS^2$ and a spherical surface $f: U \to \real^3$
that has $N$ as its Gauss map.  The solution can also be computed numerically.
Conversely, all harmonic maps and spherical surfaces
can be produced this way. 

The singularities of the harmonic
map $N$ are closely related to the lowest order (in $\lambda$) terms of the potential,
namely the pair of functions $\psi(z)= (b_0(z), c_0(z))$.
To produce $N$ from $\omega$,
one first solves the differential equation $\Phi_z = \Phi A$, with $\Phi(z_0)=I$,
then a loop group frame $\hat F$ is obtained by the (pointwise in $z$) Iwasawa decomposition (see \cite{PreS}) $\Phi(z) = \hat F(z) B(z)$ where $\hat F (z,\lambda) \in SU(2)$ for all
$\lambda \in \SSS^1$, and $B(z,\lambda)$ extends holomorphically in $\lambda$ 
to the whole unit disc. Evaluating $\hat F$ at $\lambda=1$ gives an $SU(2)$-frame $F(x,y)=\hat F(z, 1)$, (where $z=x+iy$),
and the harmonic map is given by
\[
N=\Ad_F e_3, \quad \quad e_3 = \frac{1}{2}\bbar i & 0 \\ 0 & -i \ebar \in 
\mathfrak{su}(2)=\real^3.
\]
Note that $N$ takes values in $\SSS^2$ because we use the metric
$\langle X, Y \rangle = -2 \, \hbox{trace}(XY)$, with respect to which $e_3$ is a 
unit vector. 
 A critical fact in the DPW method is that the 
map $\Phi \mapsto (\hat F, B)$ in the Iwasawa decomposition (with a suitable normalization to make the factors unique) 
 gives a real analytic diffeomorphism from the complexified loop group
 $\Lambda SL(2,\C)$ to a product of Banach Lie groups \cite{PreS}.  Using this, it is
 straightforward to verify the following:
 \begin{lemma} \label{dpwlemma1}
Let $N: U \to \SSS^2$ be a harmonic map produced by the DPW method
with potential $\omega = A(z) \dd z$.  Then the $k$-jet of $N$ is uniquely
determined by the $(k-1)$-jet of $A$.
 \end{lemma}
Moreover, the rank of the map $N$ at the integration point $z_0$ is determined
just by the 
holomorphic functions $b_0$ and $c_0$ in the potential $\omega$:
Note $\hat F = \Phi B^{-1}$, where $B$ is holomorphic in $\lambda$ on $\D$,
and we can write
 $B(z,\lambda) = \bbar \rho(z) & 0 \\ 0 & 1/\rho(z) \ebar + o(\lambda)$,
with  $\rho$ real analytic, positive real-valued, and $\rho(z_0)=1$.
It follows that
\[
\hat F^{-1} \dd \hat F = \bbar 0 & \rho^2(z) b_0(z) \\ 
                                       \rho^{-2}(z) c_0(z) & 0 \ebar 
                   \lambda^{-1} \dd z + \hbox{higher order in } \lambda.
\]
  If we write
 $F^{-1}F_z = U_{\mathfrak{k}}+ U_{\mathfrak{p}}$, where $U_\mathfrak{k}$ is parallel to $e_3$ and $U_\mathfrak{p}$ is perpendicular,
then, as usual in loop group constructions for harmonic maps, one finds that
$\hat F^{-1} \dd \hat F = U_{\mathfrak{p}} \lambda^{-1} \dd z + \hbox{higher order in } \lambda$. Hence
\beqas
N_z= \Ad_F([F^{-1} F_z, e_3]) &=& \Ad_F \left[
                            \bbar 0 & \rho^2 b_0 \\ 
                                       \rho^{-2} c_0 & 0 \ebar , e_3 \right] \\
                              &=& \Ad_F \left (-(\rho^2 b_0 + \rho^{-2} c_0) i e_2 
                                + (\rho^2 b_0 - \rho^{-2} c_0) e_1 \right).
\eeqas
It follows in a straightforward manner that $\rank(\dd N) <2$ if and only if 
$|\rho^2 b_0| = |\rho^{-2}c_0|$ and $\rank (\dd N) = 0$ if and only if $b_0 = c_0=0$.
In particular, since $\rho(z_0)=1$ we have:
\begin{lemma}  \label{dpwlemma2}
Let $N: U \to \SSS^2$ be a harmonic map produced by the DPW method
with holomorphic potential $\omega = A(z) \dd z$, integration point $z_0$, and notation
as above. Then $N$ fails to be immersed at $z_0$ if and only if $|b_0(z_0)|=|c_0(z_0)|$.
Additionally, $N$ has rank zero at any point $z \in U$ if and only if $b_0(z)= c_0(z) = 0$.
\end{lemma}

\section{Singularities of spherical surfaces}\label{sec:singSephSurf}

We are interested here in the stable singularities of 
spherical surfaces $f$ as well as those that occur generically
in 1-parameter families of such surfaces. This excludes 
the case when $\rank(\dd N)_p=0$, 
as can be deduced from Lemmas \ref{dpwlemma1} and
\ref{dpwlemma2} above together with a transversality argument.
Therefore, following Propositions \ref{prop:SSfrontfrontal} and 
Proposition \ref{prop:SSasparallel}, for the study of codimension $\le 1$ phenomena
we can consider spherical surfaces as parallels of CMC surfaces. Observe that in this case $\rank (\dd N)_p$ is never zero.

Singularities of parallels of a general surface $g:\Omega\to \mathbb R^3$  
are studied by Bruce in \cite{bruceParallel} (see also \cite{fukuihasegawa}).
Bruce considered the family of distance squared functions $F_{t_0}:\Omega\times \mathbb R^3\to \mathbb R$ 
given by $F_{t_0}((x,y),q)=|g(x,y)-q|^2-t_0^2$. 
A parallel $W_{t_0}$ of $g$ is 
the discriminant of $F_{t_0}$, that is,
$$
W_{t_0}=\{ q\in \mathbb R^3:\exists (x,y)\in \Omega 
\,\,\mbox{\rm where}\,\, 
F_{t_0}((x,y),q)=\frac{\partial F_{t_0}}{\partial x}((x,y),q)=\frac{\partial F_{t_0}}{\partial y}((x,y),q)=0 
\} .
$$

For $q_0$ fixed, the function $F_{q_0,t_0}(x,y)=F_{t_0}(x,y,q_0)$ 
gives a germ of a function at a point on the surface. Varying
$q$ and $t$ gives a $4$-parameter family of functions $F$. 
Let $\mathcal R$ denote the group of germs of diffeomorphisms from the plane to the plane.
Then, by a transversality theorem in \cite{looijenga}, for a generic surface, the possible singularities of $F_{q_0,t_0}$ 
are those of $\mathcal R$-codimension 4, and these are as follows (with  $\mathcal R$-models, up to a sign, in brackets): 
$A_1^{\pm}$ ($x^2\pm y^2$), $A_2$ ($x^2+ y^3$), 
$A_3^{\pm}$ ($x^2\pm y^4$), $A_4$ ($x^2+ y^5$) and $D_4^{\pm}$ ($y^3\pm x^2y$).  

Bruce showed that $F$ is always an $\mathcal R$-versal family 
of the $A_1^{\pm}$ and $A_2$ singularities. Consequently the parallels at such singularities are, respectively, regular surfaces or cuspidal edges. It is also shown in \cite{bruceParallel} that the transitions at an $A_2$ do not occur on parallel surfaces.

At an $A_3$-singularity one needs to consider the discriminant 
$\Delta$
of the extended family of distance squared functions  
$F:\Omega\times \mathbb R^3\times \mathbb R\to \mathbb R$, with
$F((x,y),q,t)=|g(x,y)-q|^2-t^2$, with $t$ varying near $t_0$.
The parallels $W_t$ are pre-images of the projection 
$\pi:\Delta\to \mathbb R$ to the $t$-parameter. If 
$\pi$ is transverse to the $A_3$-stratum in $\Delta$, then 
the parallels are swallowtails near $t_0$. 
If $\pi$ is not transverse to the $A_3$-stratum (we denote this case 
non-transverse $A_3^{\pm}$), the 
projection $\pi$ restricted to the $A_3$-stratum is in general a Morse function and the parallels undergo 
the transitions in the first two columns in Figure \ref{fig:familiesSS} 
(cuspidal lips or cuspidal beaks). When the function $F_{q_0,t_0}$ has an $A_4$ or  $D_4^{\pm}$-singularity and the projection $\pi$ is generic in Arnold sense \cite{ArnoldWavefront}, the transitions in the parallels 
are as in Figure \ref{fig:familiesSS}.

It is worth making an observation about the singular set of a given parallel $W_t$. The parallel is given by $f=g+tN$. Take the parameterization $g$, away from umbilic points,
in such a way that the coordinates curves are the lines of principal curvatures. Then $N_x=-\kappa_1g_x$, $N_y=-\kappa_2g_y$, so
$f_x=(1-\kappa_1t)g_x$ and $f_y=(1-\kappa_2t)g_y$. 
If the point on the surface is not parabolic, 
the parallel $W_{t}$ is singular if and only if $t={1}/{\kappa_1}$ or 
$t={1}/{\kappa_2}$. Therefore, the singular set of the parallel 
$W_{t}$ is (the image of) the curve on the surface $M$ given by ${\kappa_1}=1/t$ 
or ${\kappa_2}=1/t$. That is, the singular sets of parallels
correspond to the curves on the surface where the principal curvatures
are constant.

If the point $p$ is a parabolic point,  with say $\kappa_1(p)=0$ but $\kappa_2(p)\ne 0$, then the parallel associated to $\kappa_1$ goes to infinity and the other is singular 
along the curve $\kappa_2=1/t$. In this case, $\ker(\dd f)_p$ is parallel to the principal direction $g_y(p)$ and $\ker(\dd N)_p$ 
is parallel to the other principal direction $g_x(p)$. Observe that for generic surfaces, the parabolic curve $\kappa_1=0$ and 
the singular set 
$\kappa_2=1/t$ of the parallel $W_{t}$ 
are transverse curves at generic points on the parabolic curve.

\smallskip
We turn now to spherical surfaces $f$ and use the notation in \S \ref{sec:Prelm}. 
Since $N$ takes values in $\SSS^2$, we have $\langle f_z, N \rangle =0$ and hence it follows from 
equation \eqref{sphericalsurface} that the ranks of $f$ and $N$ coincide at every point, i.e., $
\rank (\dd f)_p = \rank (\dd N)_p.$ 
Hence, the singular set of $f$ is the same as the singular set of $N$. 
 Another way to see this is as follows. As $f$ is a parallel of a 
CMC surface, the principal curvature $\kappa_2$ is constant 
on the parabolic set $\kappa_1=0$ (or vice-versa). 
Therefore, the singular set $\kappa_2=\kappa_2(p)$ of $f$ 
coincides with the parabolic set which is the singular set of $N$.
In particular, the singularities of parallels of CMC surfaces occur only on its parabolic set. 
Observe that on the parabolic set 
$\ker(\dd f)_p$ and $\ker(\dd N)_p$ are orthogonal and 
coincide with the principal directions of the CMC surface at the point $p$.
 
We are concerned with singularities of spherical surfaces and their 
deformations within the set of such surfaces. 
We start by determining which of the singularities of a general parallel surface can occur on a spherical surface.

\begin{theorem}\label{theo:PossibCodm1SingSS}
{\rm (1)} The non-transverse $A_3^+$ and the $D_4^{\pm }$-singularities 
do not occur on spherical surfaces.

{\rm (2)} The singularities $A_2$ (cuspidal edge), $A_3$ (swallowtail), $A_4$ (butterfly) and the 
non-transverse $A_3^{-}$ (cuspidal beaks) can occur on spherical surfaces; see Figure \ref{fig:StableSS} and Figure \ref{fig:familiesSS}.
\end{theorem}

\begin{proof}

(1) The non-transverse $A_3^{+}$-singularity occurs when 
$\rank(\dd N)_p=1$ and the parabolic set has a Morse singularity $A_1^{+}$ (see Theorem \ref{theo:recogCuspidalLipsBeaks} for details). 
As the parabolic set is the singular set of the harmonic map $N$, it follows by Wood's Theorem \ref{theo:wood} that the singularity 
$A_1^+$ cannot occur for such maps. Therefore, the non-transverse $A_3^{+}$-singularity 
cannot occur for spherical surfaces. 
A $D_4^{\pm}$-singularity of a wave-front at $p$ has the
property that $\dd f_p$ vanishes. It follows from $\rank( \dd (f,N))_p=2$,
that $\dd N_p$ has rank $2$. This cannot happen for a spherical surface, because 
$\rank \dd f = \rank \dd N$.

(2) Here we first appeal to recognition criteria of singularities of wave fronts, some of which 
can be found in \cite{arnoldetal} in terms of Boardman classes. 
These criteria are expressed geometrically in \cite{izumiyasaji,izusajitakashi,suy}.
Denote by $\Sigma$ the singular set of $f$. Then the recognition criteria for wave fronts are as follows:

Cuspidal edge ($A_2$): $\Sigma$ is a regular curve at $p$; $\ker(\dd f)_p$ is transverse to $\Sigma$ at $p$ (\cite{suy}).

Swallowtail ($A_3$): $\Sigma$ is a regular curve at $p$;
$\ker(\dd f)_p$ has a second order tangency with $\Sigma$ at $p$ (\cite{suy}).

Butterfly ($A_4$): $\Sigma$ is a regular curve at $p$;
$\ker(\dd f)_p$ has a third order tangency with $\Sigma$ at $p$ (\cite{izumiyasaji}).

Cuspidal beaks (non-transverse $A_3^-$): $\Sigma$ has a Morse singularity 
at $p$; $\ker(\dd f)_p$ is transverse to the branches of $\Sigma$ at $p$ (\cite{izsata}).

We can now use the geometric Cauchy problem to construct spherical surfaces 
with the above singularities as was done in \cite{spherical}; see \S \ref{sec:constructionSS} for details. 
\end{proof}

\begin{remark}\label{rem:d2familiesCMC}
A spherical surface $f$, which we take as a parallel of a CMC $g$,  is the discriminant of the family of distance squared 
function $F_{t_0}((x,y),q)=|g(x,y)-q|^2-(1/(2H_0))^2$. 
Here, we do not have the freedom to vary $t$ as in the case for general surfaces. 
To obtain the generic deformations of wave fronts at an $A_4$ or at non-transverse $A_3^{-}$-singularity,  
we need to deform the CMC surface
in 1-parameter families $g_s$ of such surfaces in order to obtain 
a 4-parameter family $F((x,y),q,s)=|g_s(x,y)-q|^2-(1/(2H_s))^2$.
For the evolution of wave fronts at an $A_4$ (resp. non-transverse $A_3^{-}$) to be realized by spherical surfaces it is necessary
that we find a family $g_s$ of CMC surfaces such that the family $F$ is an $\mathcal R$-versal deformation of the $A_4$-singularity (resp. non-transverse $A_3^{-}$) of $F_{t_0}$. 
We establish in \S \ref{sec:GeomCritA4} geometric criteria for the family $F$
to be an $\mathcal R$-versal deformation at the above singularities 
and for the sections of the 
discriminant of $F$ along the parameter $s$ to be generic in Arnold sense \cite{ArnoldWavefront}. 
We  then use in \S \ref{sec:constructionSS} the DPW-method to construct families of spherical surfaces with the desired properties.
\end{remark}

\begin{remark}\label{rem:singSSandSinGaussmap}
The singularities of the Gauss map $N$ can be identified using 
geometric criteria involving the singular set of $N$ (i.e., the parabolic set) and $\ker(\dd N)_p$ (\cite{KabataR2R2,SajiR2R2}). 
As we observed above, 
$\ker(\dd f)_p$ and $\ker(\dd N)_p$ are orthogonal when 
the parabolic set is a regular curve, so in this case the singularities 
of $f$ and $N$ are not related. 
However, we can assert that, generically, 
a spherical surface $f$ has a cuspidal beaks singularity if and only if the 
Gauss map $N$ has a beaks singularity. The genericity condition being 
that neither  $\ker(\dd f)_p$ nor $\ker(\dd N)_p$ is tangent to the branches
of the parabolic curve.
\end{remark}


\section{Geometric criteria for $\mathcal R$-versal deformations}\label{sec:GeomCritA4}
Let now $g_s$ be any 1-parameter family of regular surfaces in $\mathbb R^3$, and 
let 
\begin{equation}\label{eq:GeneralF}
F((x,y),q,s)=|g_s(x,y)-q|^2-r(s)
\end{equation}
be a germ of a family of distance squared function on $g_s$, where 
$r$ is a smooth function (see Remark \ref{rem:d2familiesCMC} for when $g_s$ are CMC surfaces). 
We establish in this section geometric criteria for checking when $F$ 
as in (\ref{eq:GeneralF}) is an $\mathcal R$-versal deformation of an $A_4$ or a non-transverse $A_3^-$-singularity of $F_0=F_{q_0,r(0)}$.

Denote by $S_{A_k}$ the set of points $((x,y),q,s)\in \mathbb R^2\times \mathbb R^3\times \mathbb R$ 
such that $F_{q,s}$ has an $A_{\ge k}$-singularity at $(x,y)$.
Consider the following system of equations 
\small
\begin{gather}
F_{q,s}=0 \label{eq:A4Cnd1}\\ 
\frac{\partial F_{q,s}}{\partial x}=0 \label{eq:A4Cnd2}\\ 
\frac{\partial F_{q,s}}{\partial y}=0 \label{eq:A4Cnd3}\\
\frac{\partial^2 F_{q,s}}{\partial x^2}\frac{\partial^2 F_{q,s}}{\partial y^2}-\left(\frac{\partial^2 F_{q,s}}{\partial x\partial y}\right)^2=0 \label{eq:A4Cnd4}\\
\frac{\partial^3 F_{q,s}}{\partial x^3}
\left(\frac{\partial^2 F_{q,s}}{\partial x \partial y}\right)^3-
\frac{\partial^3 F_{q,s}}{\partial x^2\partial y}
\frac{\partial^2 F_{q,s}}{\partial x^2}
\left(\frac{\partial^2 F_{q,s}}{\partial x \partial y}\right)^2 
+
\frac{\partial^3 F_{q,s}}{\partial x\partial y^2}
\left(\frac{\partial^2 F_{q,s}}{\partial x^2}\right)^2
\left(\frac{\partial^2 F_{q,s}}{\partial x \partial y}\right)
-
\frac{\partial^3 F_{q,s}}{\partial y^3}
\left(\frac{\partial^2 F_{q,s}}{\partial x^2}\right)^3=0 \label{eq:A4Cnd5}
\end{gather}
\normalsize
at $(x,y)$. 
Equation (\ref{eq:A4Cnd4}) means that the quadratic part 
of the Taylor expansion of $F_{q,s}$ at a singularity $(x,y)$ is a perfect square $L^2$, and equation (\ref{eq:A4Cnd5}) means that its cubic part divides $L$. 
Then $S_{A_2}$ (resp. $S_{A_3}$) is the set of points $((x,y),q,s)$ with $F_{q,s}$ satisfying 
equations (\ref{eq:A4Cnd1}) - (\ref{eq:A4Cnd4}) (resp. (\ref{eq:A4Cnd1}) - (\ref{eq:A4Cnd5})) 
at $(x,y)$.

\subsection{The $A_4$-singularity}\label{subsec:A4}

\begin{theorem}\label{theo:GemCritVersA4}
The family $F$ in (\ref{eq:GeneralF})
is an $\mathcal R$-versal deformation of an $A_4$-singularity of $F_0$ at $p_0$ if and only if 
the $S_{A_3}$ set is a regular curve at $p_0$.  When this is the case, the sections of the discriminant of $F$
along the parameter $s$ are generic if and only if the projection of the $S_{A_3}$ curve to 
the $(x,y)$ domain is a regular curve (then we get the bifurcations in Figure \ref{fig:familiesSS}, third column, 
in the wave fronts $W_{r(s)}$ as $s$ varies near zero).
\end{theorem}

\begin{proof}
	We 
	take, without loss of generality, the family of surfaces in Monge form $g_s(x,y)=(x,y,h_s(x,y))$ 
	at $p_0=(0,0)$ and write $q=(a,b,c)$. 
	
	We write the homogeneous part of degree $k$ in the Taylor expansion of $h_0$ as
	$\sum_{i=0}^ka_{i,k-i}x^iy^{k-i}$ and set
	$a_{0,0}=a_{1,0}=a_{0,1}=0$.
	
As we want the origin to be a singularity of $F_0$, we take $q_0=(0,0,c_0)$ and $r(0)=c_0^2$ 
so that  $F_0(0,0)=0$. We can make a rotation of the coordinate system and set $a_{1,1}=0$. 
	Then $a_{0,2}-a_{2,0}\ne 0$ as the origin is not an umbilic point. In particular, 
	$a_{0,2}\ne 0$ or $a_{2,0}\ne 0$. We suppose $a_{2,0}\ne 0$ and
	take $c_0=1/(2a_{0,2})$ in order for $F_0$ to have an $A_{\ge 2}$-singularity at the origin. Then 
	the conditions for the origin to be an $A_4$-singularity of $F_0$ are:
	$$
	\begin{array}{c}
	a_{0,3}=0,\,\,\, a_{1,2}^2+4(a_{0,2}-a_{2,0})(a_{0,4}-a_{0,2}^3)=0,\\
	4(a_{0,2}-a_{2,0})^2a_{0,5}+a_{1,2}(a_{2, 1}a_{1, 2}+2a_{1, 3}a_{0, 2}-2a_{1, 3}a_{2, 0})\ne 0.
	\end{array}
	$$
	
	Denote by
	$\dot{F}_a=({\partial F}/{\partial a})|_{a=0,b=0,c=c_0,s=0}$
	(similarly for $\dot{F}_b$, $\dot{F}_c$ and $\dot{F}_s$). Let 
	${\mathcal E}(2,1)$ be the ring of germs of functions $(\mathbb R^2,0)\to \mathbb R$ and $\mathcal M_2$ its maximal ideal.
	The family $F$ is an $\mathcal R$-versal deformation of the singularity of $F_0$ if and only if 
	
	\begin{equation}\label{eq:CondFintDetA4}
	{\mathcal E}(2,1)\{
	\frac{\partial F_{0}}{\partial x},
	\frac{\partial F_{0}}{\partial y}\}+
	\mathbb R\cdot{}
	\{\dot{F}_a,\dot{F}_b,\dot{F}_c,\dot{F}_s\}=\mathcal E(2,1).
	\end{equation}
	(see, e.g., \cite{Martinet}). 
	As $F_0$ is $5$-$\mathcal R$-determined, it is enough to show 
	that \eqref{eq:CondFintDetA4} holds modulo $\mathcal M^6_2$, i.e., 
	we can work in the $5$-jet space $J^5(2,1)$. Write  
	$j^5h_s=j^5h_0+j^5\left(\sum_{i,j,k} \beta_{i, j, k-i-j}x^iy^js^{k-i-j}\right)s$. Then,
	using ${\partial F_0}/{\partial x}$, 
	${\partial F_0}/{\partial y}$ and $\dot{F}_b$, we can get all the 
	monomials in $x,y$ of degree $\le 5$ in the left hand side of \eqref{eq:CondFintDetA4} except $1,x,y^2,y^3$. 
	Now, we can write $\frac{\partial F_0}{\partial x}$, $\dot{F}_a$,$\dot{F}_b$,$\dot{F}_c$,$\dot{F}_s$
	modulo the monomials already in the left hand side of \eqref{eq:CondFintDetA4}
	as linear combinations of $1,x,y^2,y^3$:
	$$
	\begin{array}{rl}
	\frac{\partial F_{0}}{\partial x}\sim& 
	2x-\frac{a_{1,2}a_{2,1}+a_{1,3}(a_{0,2}-a_{2,0})}{(a_{0,2}-a_{2,0})^2}y^3,\\
	\dot{F}_a\sim& 2x+\frac{a_{1,2}}{a_{0,2}-a_{2,0}}y^2,\\
	\dot{F}_c\sim& \frac{1}{a_{0,2}}-2a_{0,2}y^2,\\
	\dot{F}_s\sim& 
	(a_{0, 2}r'(0)-\beta_{0, 0, 0})-\beta_{1, 0, 0}x
	+\frac{2(a_{0, 2}-a_{2, 0})(2a_{0, 2}^2\beta_{0, 0, 0}-\beta_{0, 2, 0})	-a_{1, 2}\beta_{1, 0, 0}}
	{2(a_{0,2}-a_{2,0})}y^2\\
	&
	+\frac{
		2(a_{0, 2}	-a_{2,0})(2a_{0,2}^2\beta_{0, 1, 0}-\beta_{0, 3, 0})-a_{1, 2}\beta_{1, 1, 0}
	}
	{2(a_{0,2}-a_{2,0})}y^3.
	\end{array}
	$$
	
	The family $F$ is an $\mathcal R$-versal deformation if and only if 
	the above vectors are linearly independent, equivalently,
	$$
	\begin{array}{l}
	4a_{0, 2}^3(a_{1, 2}a_{2, 1}+a_{1, 3}(a_{0, 2}-a_{2, 0}))r'(0)
	-4a_{0, 2}^2a_{1,2}(a_{0, 2}-a_{2, 0})\beta_{0, 1, 0}\\
	+a_{1, 2}^2\beta_{1, 1, 0}-2(a_{1, 2}a_{2, 1}+a_{1, 3}(a_{0, 2}-a_{2, 0}))\beta_{0, 2, 0}
	+2a_{1, 2}(a_{0, 2}-a_{2,0})\beta_{0, 3, 0}\ne 0.
	\end{array}
	$$

We consider now the $S_{A_3}$ set. Equations (\ref{eq:A4Cnd2}) - (\ref{eq:A4Cnd4}) give $a,b,c$ as functions of $x,y,s$. Substituting these in equations (\ref{eq:A4Cnd1}) and (\ref{eq:A4Cnd5}), 
gives the following 1-jets of their left hand sides, respectively, up to non-zero scalar multiples, 
\small
\begin{gather}
	-a_{1, 2}x+(2a_{0, 2}^3r'(0)-\beta_{0,2,0})s,\nonumber\\
	2(a_{1, 3}a_{2, 0}-a_{1, 3}a_{0, 2}-a_{1, 2}a_{2, 1})x+
	\left(2(a_{0, 2}	-a_{2,0})(2a_{0,2}^2\beta_{0, 1, 0}-\beta_{0, 3, 0})-a_{1, 2}\beta_{1, 1, 0}\right)s. \label{eq:1jetQ1}
\end{gather}
\normalsize
The set $S_{A_3}$ is a regular curve if and only if the above 1-jets are linearly independent, 
which is precisely the condition above for $F$ to be an $\mathcal R$-versal family.
	
	It is not difficult to show that we get the generic sections of the discriminant of $F$ (see \cite{ArnoldWavefront}) if and 
	only if the coefficient of $y^3$ in $\dot{F}_s$ above is not zero. This is precisely the condition for 
	for coefficient of $s$ in (\ref{eq:1jetQ1}) to be non-zero, 
	which in turn is equivalent to the projection of the curve $S_{A_3}$ to the 
	$(x,y)$ domain to be regular.
\end{proof}
\subsection{The non-transverse $A_3^{-}$}\label{subsec:nonversalA3}

Bruce gave in \cite{bruceParallel} geometric conditions for parallels of a surface $M$ in 
$\mathbb R^3$ to undergo the generic bifurcations of wave fronts in 
\cite{ArnoldWavefront} at a non-transverse $A_3^{\pm}$. 
In \cite{bruceParallel}, $r(s)=s_0+s$ in (\ref{eq:GeneralF}) and $g_s=g_0$, so $F$ is an $\mathcal R$-versal  
deformation and  $S_{A_2}$ (resp.  $S_{A_3}$) is a regular surface (resp. curve).
For $F$ as in  (\ref{eq:GeneralF}), the geometric conditions  
in \cite{bruceParallel} are:
(i) $F$ is an $\mathcal R$-versal deformation of $F_0$, so $S_{A_2}$ (resp.  $S_{A_3}$) is a regular surface (resp. curve)) and (ii) the projection 
$\pi:S_{A_2} \to (\mathbb R,0)$ to the parameter 
 $s$ is a submersion and its restriction to the $S_{A_3}$ curve is a Morse function (i.e., $\pi^{-1}(0)$ and $S_{A_3}$ have ordinary tangency). Then the parallels $W_{r(s)}$ undergo the bifurcations in Figure \ref{fig:familiesSS}, second column. (When $\pi^{-1}(0)$ is transverse to the
 $S_{A_3}$ curve, i.e., $F_0$ has a transverse $A_3^{\pm}$-singularity, 
 the parallels $W_{r(s)}$ are all swallowtails for $s$ near zero.) We give below equivalent geometric conditions 
 which are useful for constructing families of spherical surfaces in \S \ref{sec:constructionSS} 
 with the desired properties.

\begin{theorem}\label{theo:recogCuspidalLipsBeaks}
(1)  The family $F$ as in (\ref{eq:GeneralF}) is an $\mathcal R$-versal deformation of 
 a non-transverse $A_3^{\pm}$-singularity of $F_0$ if and only if the projection $\pi$ is a submersion,
 equivalently, the surface formed by the singular sets of $W_{r(s)}$ in the $(x,y,s)$-space is a regular surface.
 
(2) Suppose that (1) holds.  Then, the projection $\pi|_{S_{A_3}}$ is 
a Morse function if and only if the singular set of the
parallel $W_{r(0)}$, in the domain, has a Morse singularity.

(3) As a consequence, the wave fronts
$W_{r(s)}$ undergo the bifurcations in Figure \ref{fig:familiesSS}, second column
if and only if 
the surface formed by the singular sets of $W_{r(s)}$ in the $(x,y,s)$-space is a regular surface 
and its sections by the planes $s=constant$ 
undergo the generic Morse transitions as $s$ varies near zero.
\end{theorem}

\begin{proof}
(1) We take the surfaces $g_s$ in Monge form $g_s(x,y)=(x,y,h_s(x,y))$ 
and use the notation in the proof of Theorem \ref{theo:GemCritVersA4}. 
For $h_0$,  we have $a_{2,0}-a_{0,2}\ne 0$ and we can take $a_{0,2}\ne 0$. 
Then $F_{0}$  has a non-transverse $A_3^{\pm}$-singularity at the 
origin if and only if 
$
a_{0,3}=0$, $a_{1,2}=0$  and $a_{0,2}^3- a_{0,4}\ne 0.$

Calculations similar to those in the proof of Theorem \ref{theo:GemCritVersA4}
show that $F$ is an $\mathcal R$-versal deformation of the non-transverse $A_3^{\pm}$-singularity of $F_0$ at the origin 
if and only if $2a_{0, 2}^3r'(0)-\beta_{0, 2, 0}\ne 0$. 

We calculate the set $S_{A_2}$  in the same way as in the proof of Theorem \ref{theo:GemCritVersA4} by 
solving the system of equations \eqref{eq:A4Cnd1}-\eqref{eq:A4Cnd4}. We use 
\eqref{eq:A4Cnd2}-\eqref{eq:A4Cnd4} to write $a,b,c$ as functions of 
$(x,y,s)$ and substitute in \eqref{eq:A4Cnd1} to obtain  
\begin{equation}\label{eq:SA2nontransverseA3}
(2a_{0, 2}^3r'(0)-\beta_{0, 2, 0})s-Q(x,y) +\alpha_1 sx+\alpha_2s^2+O_3(x,y,s)=0,
\end{equation}
with $\alpha_1,\alpha_2$ irrelevant constants,  $O_3$ a smooth function with a zero 2-jet and
\small
\begin{equation}\label{eq:2jetSA2s=0}
Q(x,y)=\frac{2a_{02}^2a_{20}^2-2a_{02}a_{20}^3-a_{02}a_{22}+a_{20}a_{22}-a_{21}^2}{a_{02}-a_{20}}x^2
-3a_{13}xy+6(a_{02}^3-a_{04})y^2.
\end{equation}
\normalsize

Clearly,  \eqref{eq:SA2nontransverseA3} shows that the condition $2a_{0, 2}^3r'(0)-\beta_{0, 2, 0}\ne 0$ 
for $F$ to be an $\mathcal R$-versal family
is precisely the condition for $\pi$ to be a submersion 
and for the projection of the set $S_{A_2}$ to the $(x,y,s)$-space to be a regular surface 
(which is the surface formed by the family of the singular sets of the wave fronts $W_{r(s)}$). 

For (2), the singular set of $W_{r(0)}$, in the domain, is obtained by setting $s=0$ in \eqref{eq:SA2nontransverseA3}.
We can always solve \eqref{eq:A4Cnd5} for $y$. Substituting in  
\eqref{eq:SA2nontransverseA3}
gives the equation of the set $S_{A_3}$ in the form 
$$
(2a_{0, 2}^3r'(0)-\beta_{0, 2, 0})s- \frac{\Lambda}{24(a_{02}^3-a_{04})}
x^2+O_3(x,s)=0,
$$
with $\Lambda$ the discriminant of the quadratic form $Q$ in \eqref{eq:2jetSA2s=0}, and the result follows.

Statement (3) is a consequence of (1) and (2) and application of the results from\cite{bruceParallel}. 
\end{proof}

\section{Construction of spherical surfaces and the singular geometric Cauchy problem}\label{sec:constructionSS}
In this section we show how to construct all the codimension $\le 1$ singularities as well as their 
bifurcations in generic $1$-parameter families of spherical surfaces, from simple geometric data along 
a regular curve in $\SSS^2$.

Singularities of spherical surfaces are analyzed in \cite{spherical}, using $SU(2)$ frames for the surfaces.
The $SU(2)$ frames are needed in order to construct the solutions using loop group methods, however
they are not needed to discuss singularities in terms of geometric data along the curve. We therefore
begin with a more direct geometric derivation of some of the results of \cite{spherical}.

As mentioned above, for a spherical frontal $f$, with corresponding Gauss map $N$,
we have $\rank (\dd f) = \rank (\dd N)$, so the singular set of $f$ is the same as the
singular set of $N$ and is determined by the vanishing of the function
\beqas
\mu = \langle f_x \times f_y , N \rangle = \langle  N_x \times N_y , N \rangle.
\eeqas

A singular point is called \emph{non-degenerate} if $\dd \mu \neq 0$ at that point (see \cite{suy}).

\subsection{Stable singularities}
If a frontal is non-degenerate at a point then the local singular locus is a regular curve in the coordinate domain.
In such a case we can always choose local conformal coordinates $(x,y)$ such that the singular set is locally 
represented by the curve \mbox{$\{ y=0 \}$} in the domain and by $f(x,0)$ on the spherical surface.  We make this assumption throughout this section.  To analyze the singular
curve in terms of geometric data, it is also convenient to use an arc-length parameterization.
Considering the pair of equations
\beq  \label{fNeqns}
f_x = N \times N_y, \quad \quad f_y = -N \times N_x,
\eeq
it is clear that, at a given singular point, either $f_x \neq 0$ or $N_x \neq 0$ (or both). This allows us to use either the
curve $f(x,0)$ or the curve $N(x,0)$ as the basis of analysis, since at least one of them is regular.

\subsubsection{Case that $f(x,0)$ is  a regular curve }  In this case, we assume further that coordinates are chosen
such that $|f_x(x,0)|=1$.  An orthonormal frame for $\real^3$ is given by
\[
e_1 = \frac{f_x}{|f_x|}, \quad \quad e_2 = e_3 \times e_1 = -\frac{N_y}{|f_x|},  \quad \quad e_3 = N,
\]
where the expression for $e_2$ comes from \eqref{fNeqns}.
Now let us write 
\[
f_y= ae_1 + be_2.
\]
Then $f_x \times f_y =  \mu e_3$, where $\mu = |f_x|b$,
and the singular curve $C$ in the domain 
is locally given by $b=0$.  Along $C$, we have
\[
\dd \mu = 0 + 1 \cdot \dd b =  \dd b = \frac{\partial b}{\partial y} \dd y,
\]
so the non-degeneracy condition is $b_y \neq 0$.  From \eqref{fNeqns}, we also have
\[
N_x = -b e_1 + a e_2.
\]
To find the Frenet-Serret frame $(\bt, \bn, \bb)$, along the curve $f(x,0)$,
we can differentiate $f_x = N \times N_y$ to obtain
\beqas
\kappa \bn = f_{xx} & = & N_x \times N_y + N \times N_{xy} \\
&=& N \times N_{xy},
\eeqas
because $N_x$ and $N_y$ are parallel along the singular curve. Here $\kappa(x)$ is the curvature
function of the curve $f(x,0)$.  Hence $\bn$ is orthogonal to $N$,
and so we conclude that
\[
(\bt, \bn, \bb) = (e_1 (x,0), e_2(x,0), e_3(x,0)).
\]
Then the Frenet-Serret formulae gives
\[
- \tau \bn = \bb_x = N_x.
\]
Comparing with the expression above for $N_x$ we conclude that 
\[
a = -\tau.
\]
Thus, along $C$, we have $f_x = e_1$ and $f_y = - \tau e_1$. Hence the \emph{null direction} for $f$, i.e. the kernel
of $\dd f$ is given by 
\[
\eta_f = \tau \partial _x + \partial _y.
\]
Since the null direction is transverse to the singular curve, it follows by using the recognition criteria in \cite{suy} that the singularity is a cuspidal edge.

\begin{figure}[ht]
	\centering
	$
	\begin{array}{ccc}
	\includegraphics[height=22mm]{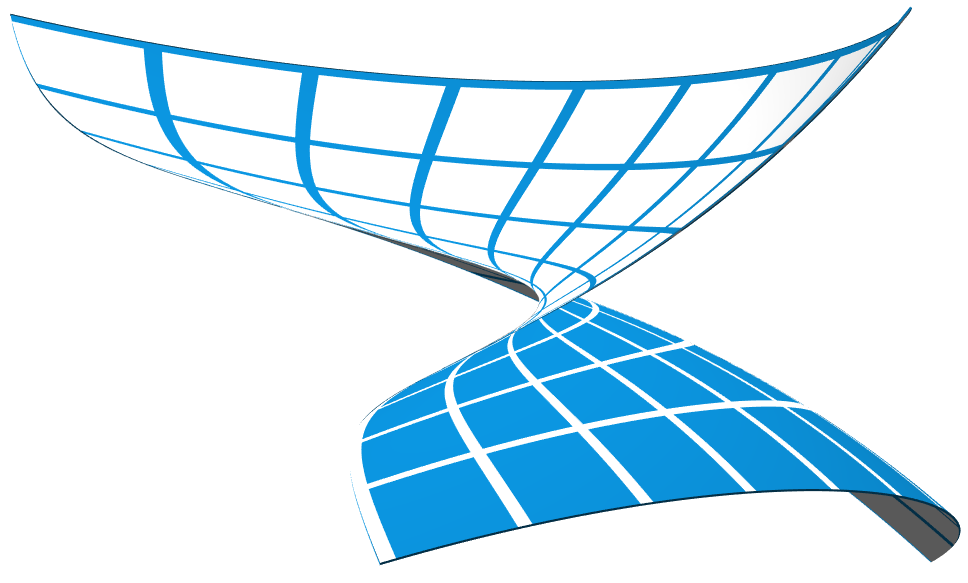}  \quad & 
	\includegraphics[height=22mm]{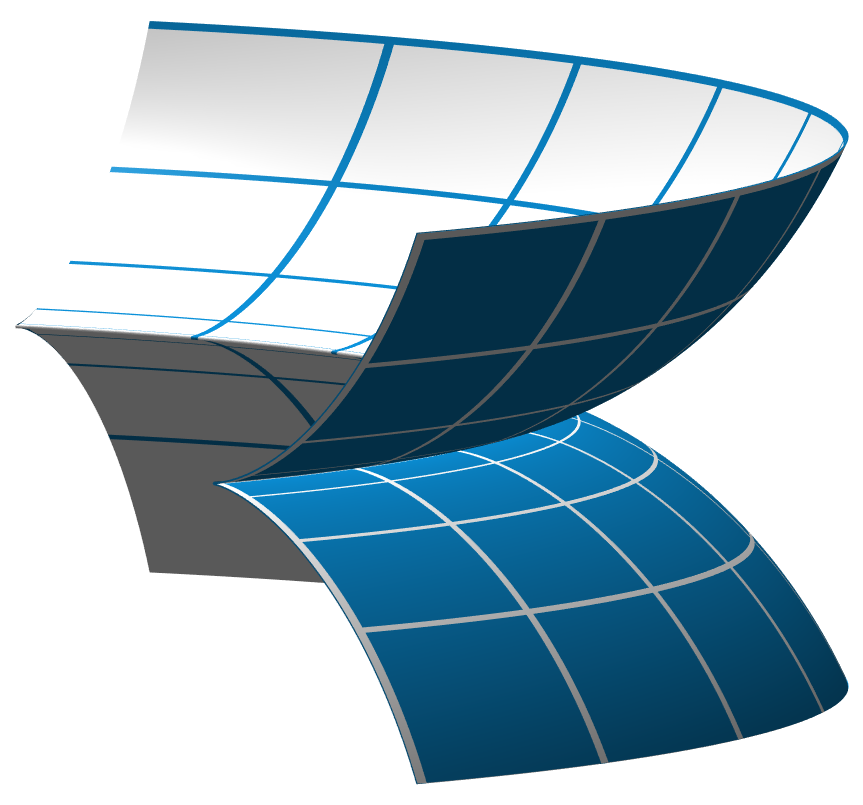} & \quad
	\includegraphics[height=22mm]{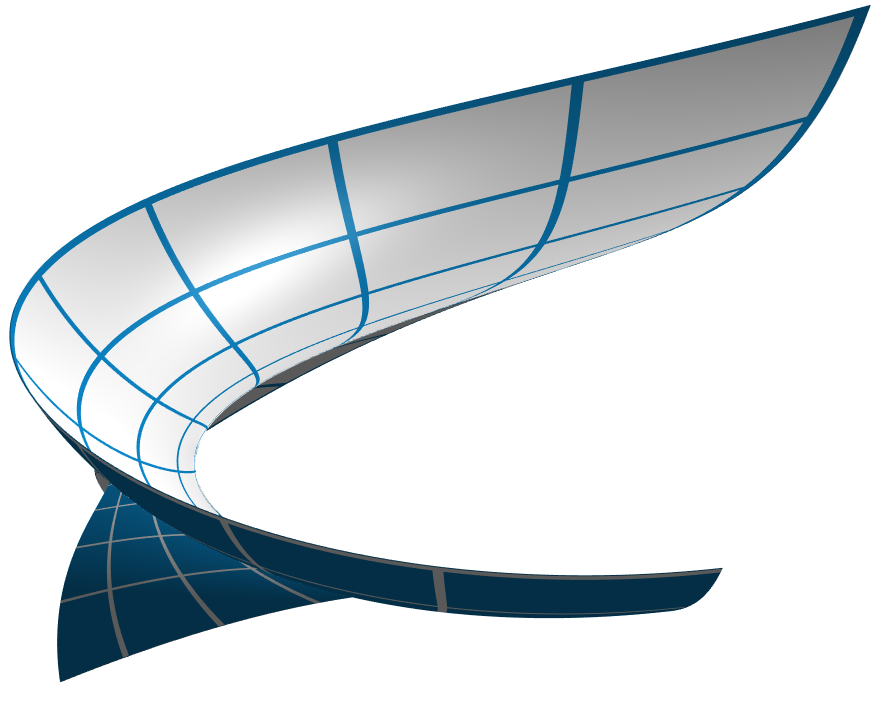}   \vspace{1ex}\\
	\includegraphics[height=22mm]{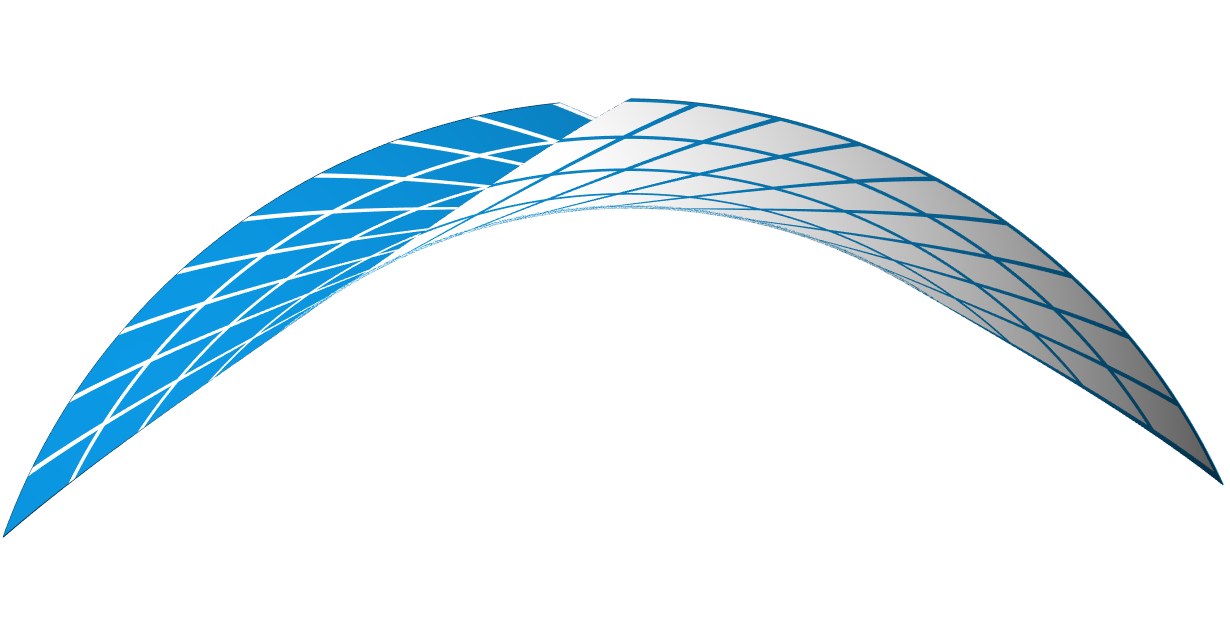}  \quad & 
	\includegraphics[height=22mm]{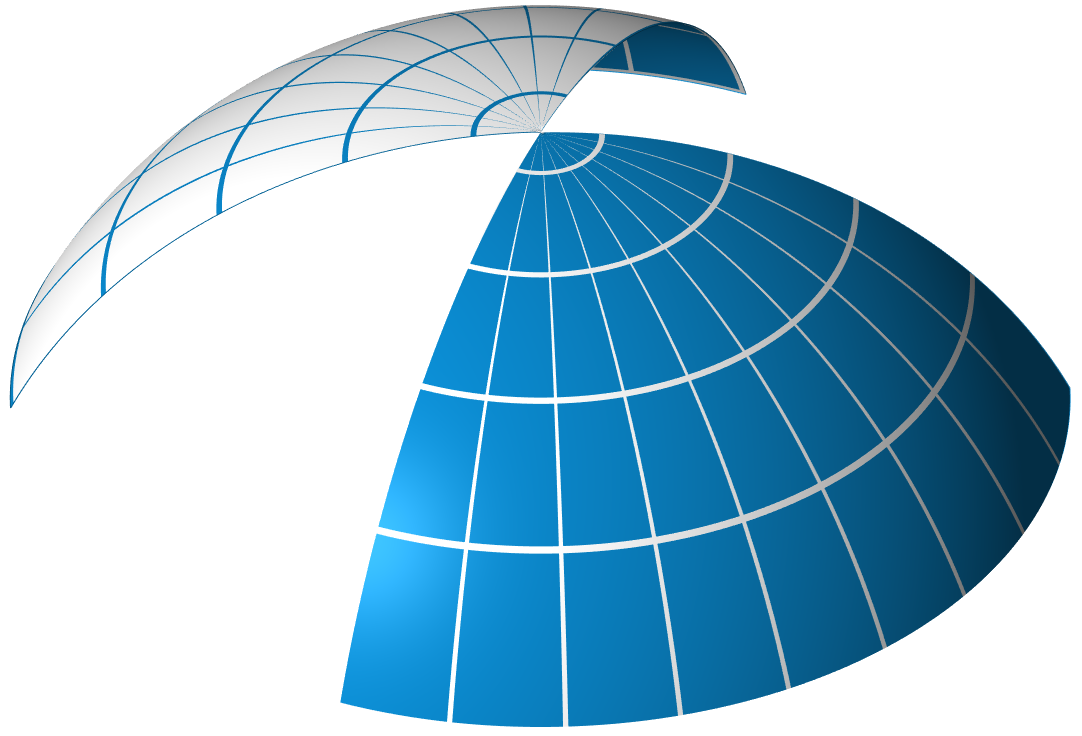}  & \quad
	\includegraphics[height=22mm]{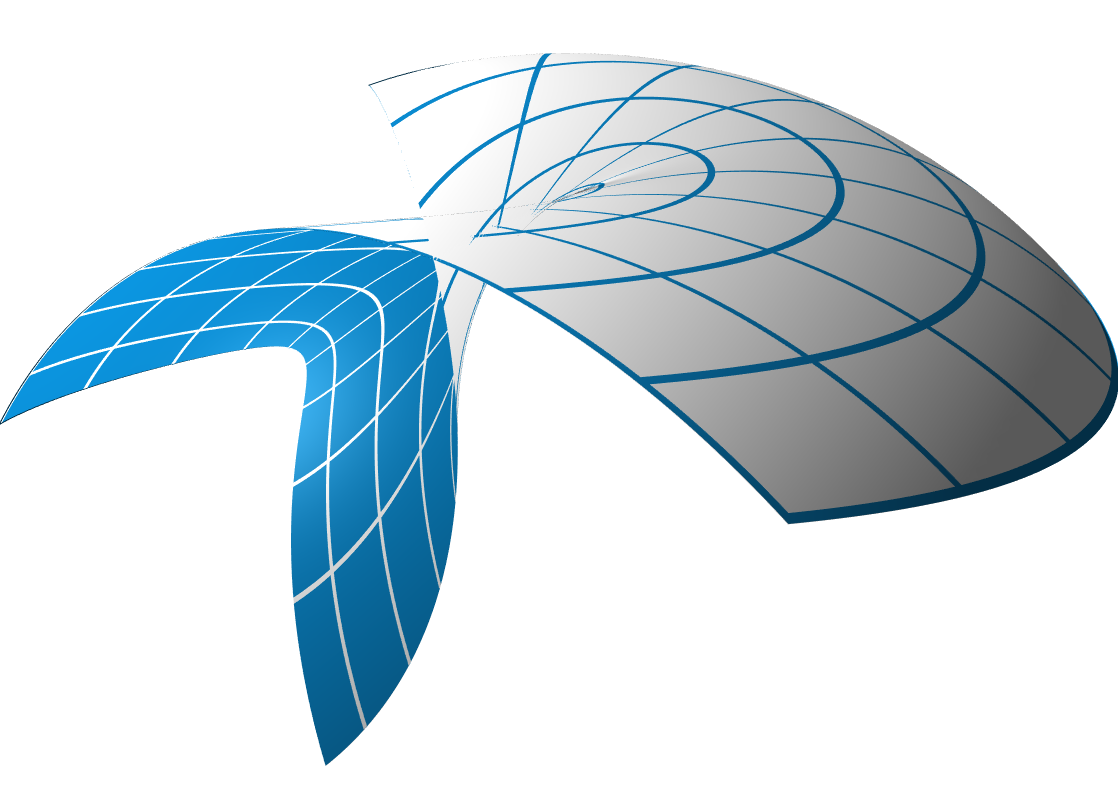}  \\
	(\kappa, \tau)=(1,1) &  (\kappa,\tau)=(1,0) & (\kappa,\tau)=(1,s)
	\end{array}
	$
	\caption{Above: the spherical surface $f$ generated by the singular curve data $(\kappa(s),\tau(s))$.
		Below: the corresponding harmonic Gauss map $N$.
	}
	\label{figure_ex1}
\end{figure}

Note: the nondegeneracy condition $\partial_y b \neq 0$ can be given in terms of the curve $f(x,0)$ as follows.
We have $b=-\langle N_x, e_1 \rangle$, so 
\[
\frac{\partial b}{\partial y} = - \langle N_{yx}, e_1 \rangle - \langle N_x, \partial _y e_1 \rangle.
\]
Along $C$ we have $N_x = -\tau e_2$ and $N_y = -e_2 = f_x \times N$, so
\[
N_{yx} = \kappa e_1 - \tau e_3.
\]
Differentiating $e_1 = f_x/|f_x|$ and restricting to the curve $f(x,0)$ along which $|f_x(x,0)|=1$ , we have
\beqas
\frac{\partial e_1 }{\partial y} = \frac{\partial}{\partial y}\left(\frac{1}{|f_x|}\right) e_1 + f_{xy}
&=& M e_1 - \tau \kappa e_2,
\eeqas
where the factor $M$ is immaterial. Hence
\[
\frac{\partial b}{\partial y}  = -\kappa(1+\tau^2),
\]
and so the non-degeneracy condition is $\kappa \neq 0$.

\medskip
\noindent 
\textbf{Comparison with the singular set of $N$:}
Along $C$ we have $N_x = -\tau e_2$ and $N_y = -e_2$, so the null direction for $N$ is
\[
\eta_N = \partial _x - \tau \partial y.
\]
Therefore it is possible for the null direction for $N$ to point along the curve.  According to the terminology in 
\cite{wood1977}, the harmonic map $N$ has, at $p=(x_0,0)$:
\begin{enumerate}
	\item A fold point if $\tau(x_0) \neq 0$;
	\item A collapse point if $\tau(x) \equiv 0$, on a neighbourhood of $x_0$;
	\item A good singularity of higher order otherwise.
\end{enumerate}
Note that these singularities of $N$ all correspond to cuspidal edge singularities of $f$.
Geometrically, they are reflected in the geometry of the spherical frontal by whether or not
the singular curve is planar, or has a single point of vanishing torsion (see Figure \ref{figure_ex1}).

\subsubsection{Case that $N(x,0)$ is a regular curve } 
In this case we can choose coordinates such that the curve $N(x,0)$ is unit speed,
so $|N_x(x,0)|= |f_y(x,0)| = 1$, and a frame
\[
e_1 = \frac{N_x}{|f_y|}, \quad e_2 = -\frac{f_y}{|f_y|}, \quad e_3 = N.
\]
Write 
\[
f_x = a e_1 + b e_2,
\]
so $f_x \times f_y = \mu e_3$, where $\mu = -a |f_y|$, 
so the singular curve $C$ in the domain is the set $a=0$, 
and the non-degeneracy condition is
\[
\frac{\partial a}{\partial y} \neq 0.
\]
From $f_x = N \times N_y$ we have
\[
N_y = b e_1 - a e_2.
\]
Since, along $C$, we have $f_x=be_2$ and $f_y=-e_2$, and $N_x = e_1$ and $N_y = b e_1$, the
null directions for the two maps are:
\beqas
\eta_f = \partial _x + b \partial _y,\\
\eta_N = b \partial _ x - \partial _y.
\eeqas
Hence the singularity for $N$ is always a fold point, because $\eta_N$ is transverse to the singular curve.

For the map $f$, which is a wave front, we can use the criteria from \cite{krsuy} to conclude that
the singularity at the point $(x_0,0)$ is: 
\begin{enumerate}
	\item A cone point if and only if $b(x,0) \equiv 0$ in a neighbour of $x_0$; 
	\item A cuspidal edge if $b(x_0,0) \neq 0$.
	\item A swallowtail if and only if $b(x_0,0)=0$ and $b_x(x_0,0) \neq 0$;
	\end{enumerate}
Examples are shown in Figure \ref{figure_ex2}.
\begin{figure}[ht]
	\centering
	$
	\begin{array}{cccc}
	\includegraphics[height=22mm]{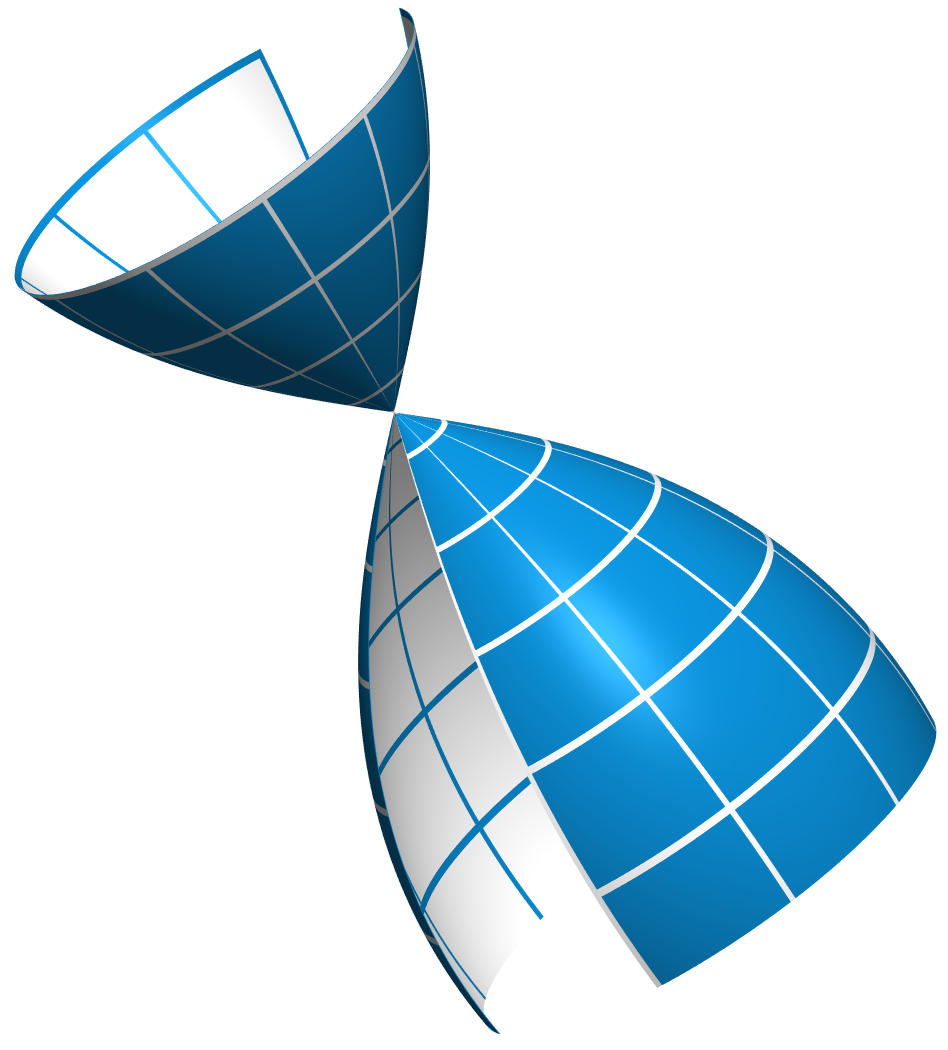} \,\, & \, \,
	\includegraphics[height=22mm]{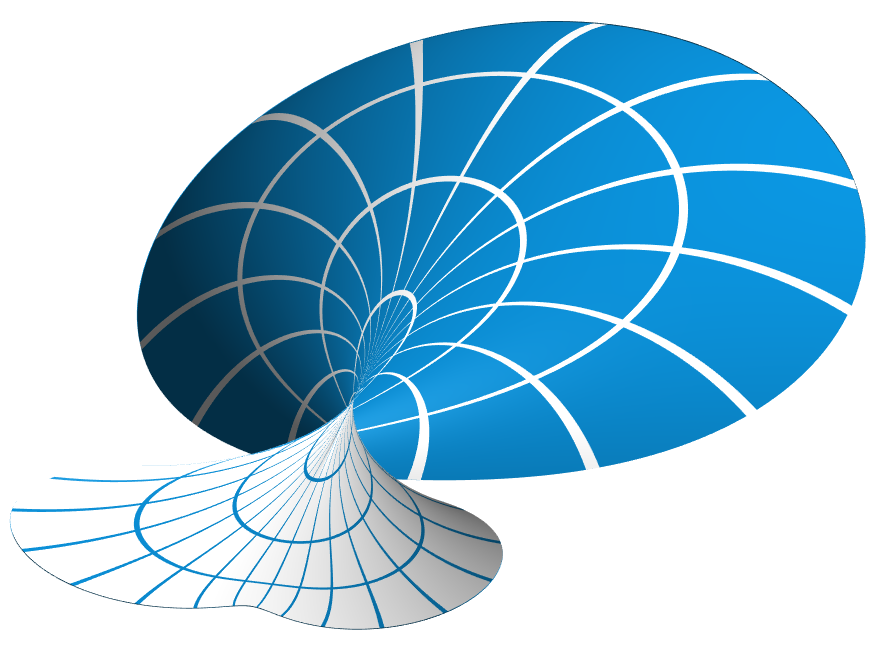}  \,\, & \, \,
	\includegraphics[height=22mm]{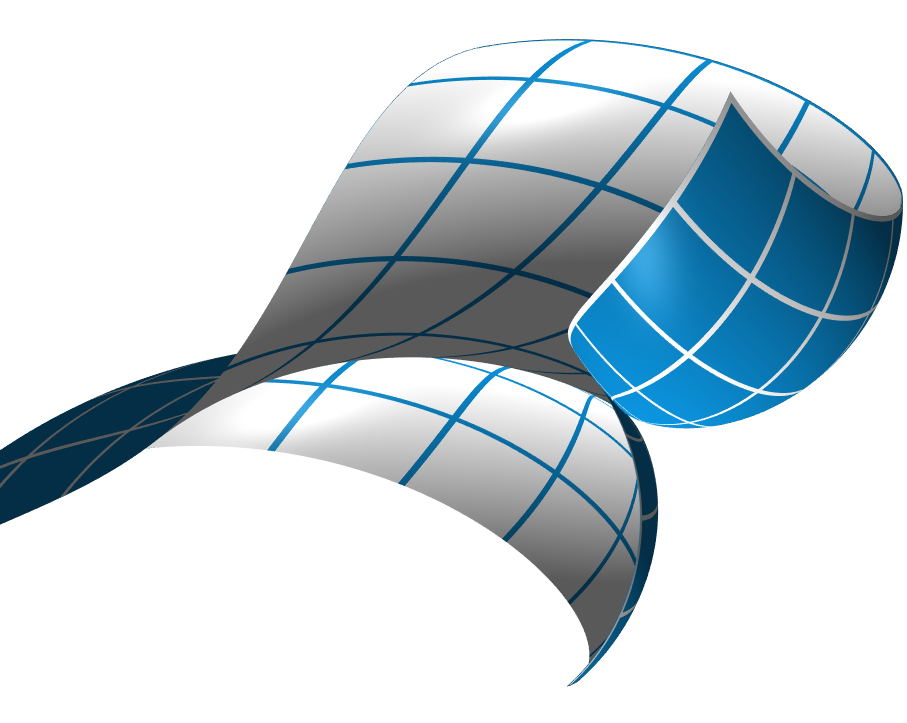}  \,\, & \, \,
	\includegraphics[height=22mm]{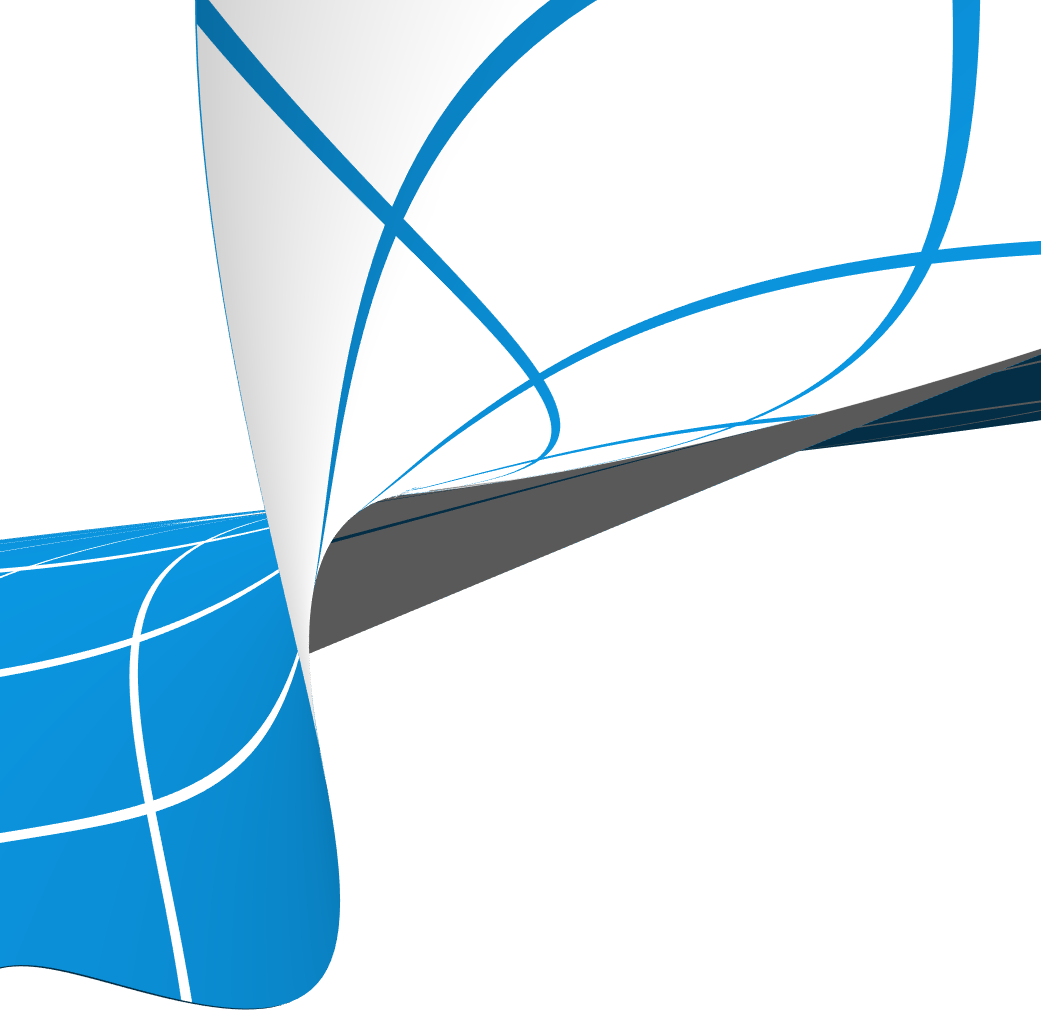}   \vspace{1ex} \\
	\includegraphics[height=22mm]{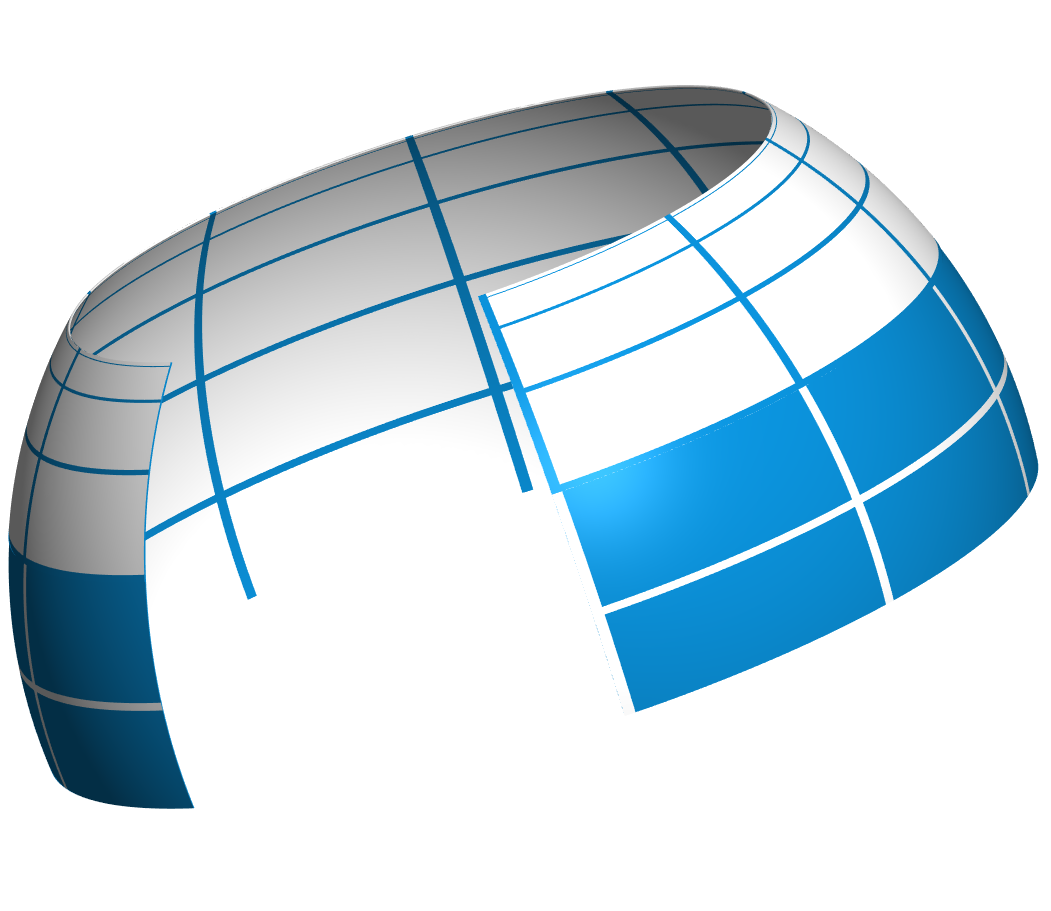}  \,\, & \, \,
	\includegraphics[height=22mm]{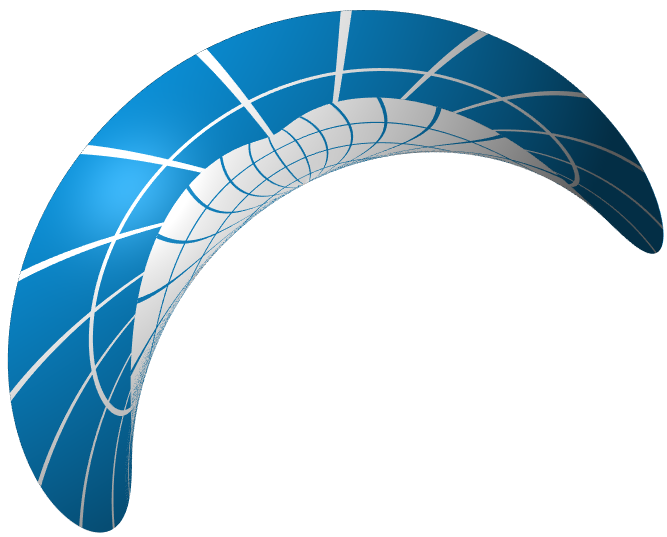}  \,\, & \, \,
	\includegraphics[height=22mm]{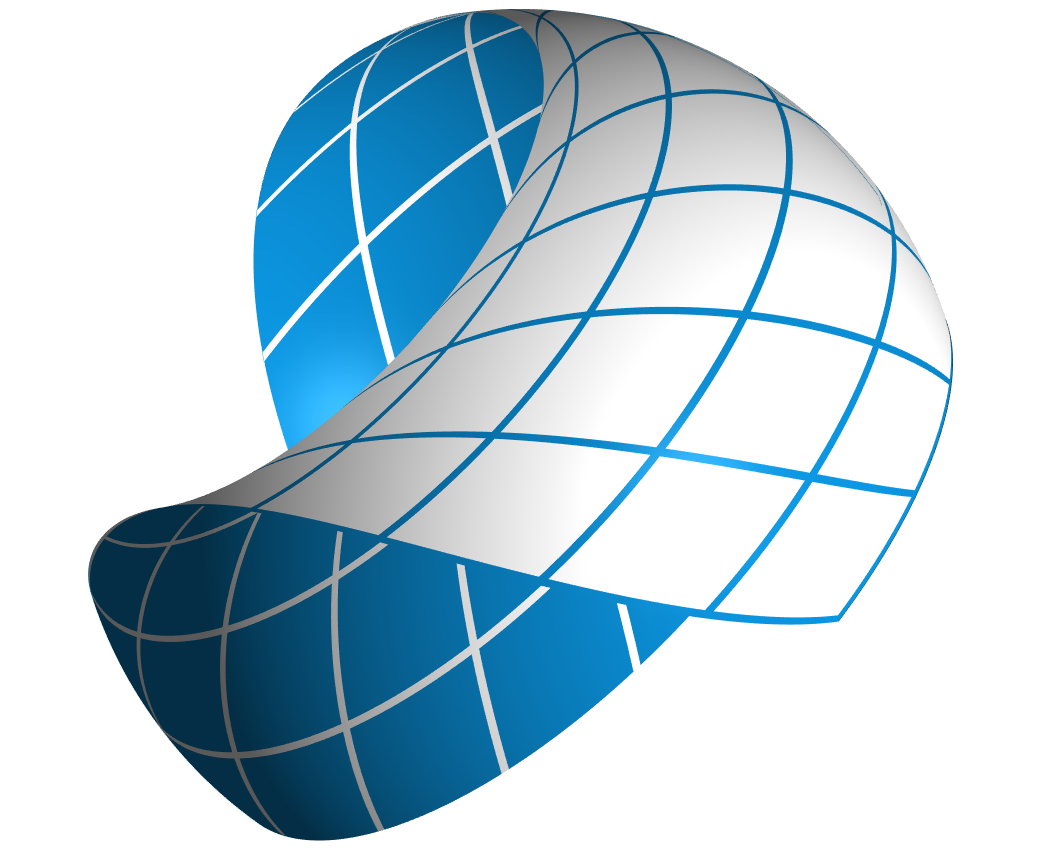}  \,\, & \, \,
	\includegraphics[height=22mm]{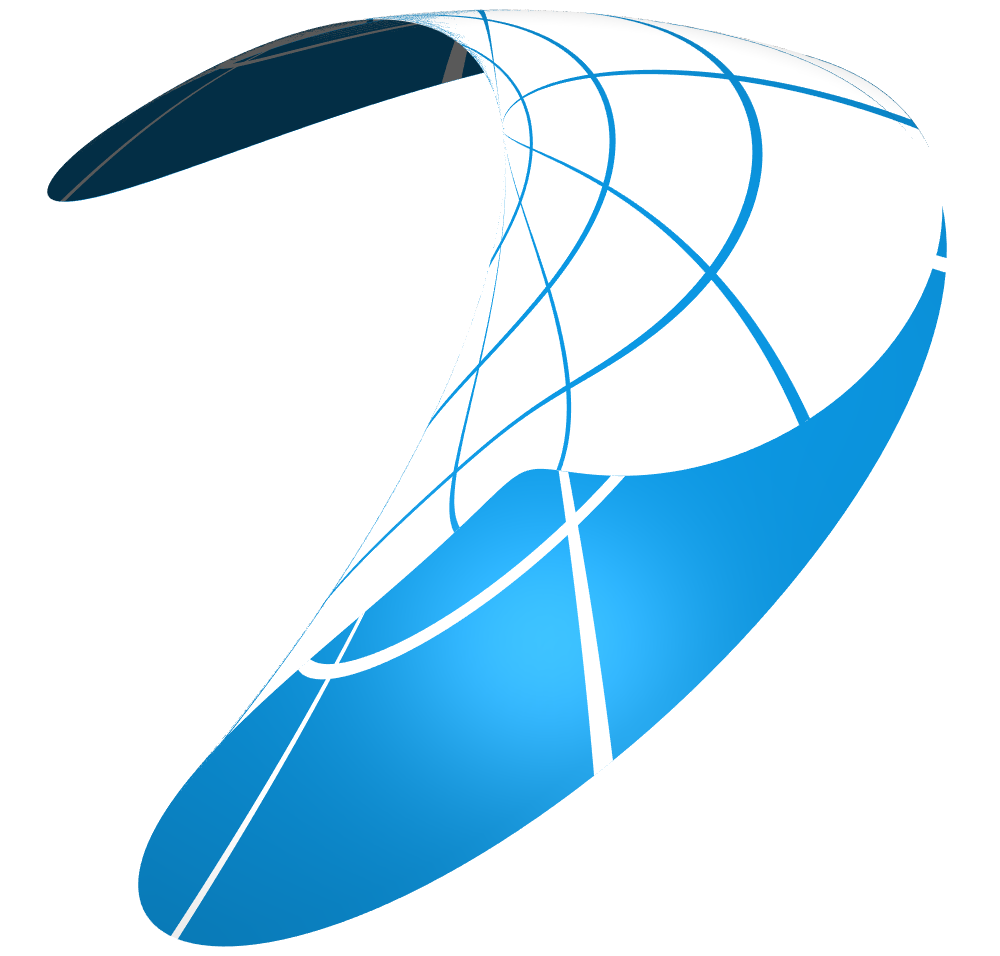}  \\
	b=0 &  b=t & b=1 & b=t^2
	\end{array}
	$
	\caption{
		Above: the spherical surface $f$ generated by non-degenerate singular curve data, with $b(t,0)$ as indicated.
		Below: the corresponding harmonic Gauss map $N$.
	}
	\label{figure_ex2}
\end{figure}

Let's consider the above in terms of the geometry of the curve $\alpha(x) = N(x,0)$ in $\SSS^2$.
Along the curve $\alpha$, the Darboux frame is $e_1 = N_x$, $e_2 = N \times N_x = -f_y$, $e_3 = N$,
and the Darboux equations reduce to (writing the geodesic curvature of $\alpha$ in $\SSS^2$ as
$\hat \kappa_g$),
\beqas
N_{xx} &=& \hat \kappa_g e_2 - e_3, \\
-\partial _x e_2 &=& f_{yx} = \hat \kappa_g N_x = \hat \kappa_g e_1.
\eeqas
For the non-degeneracy condition $a_y \neq 0$, we have $a = \langle f_x, e_1 \rangle$,
so 
\[
\frac{\partial a}{\partial y } = \langle f_{yx}, e_1 \rangle + \langle f_x, \frac{\partial e_1}{\partial y} \rangle.
\]
The first term on the right is $\langle \hat \kappa_g e_1, e_1 \rangle = \hat \kappa_g$, and for the second
compute (along $C$)
\beqas
\frac{\partial e_1}{\partial y} &=& \partial _y \left(\frac{1}{|N_x|} \right) e_1 + N_{yx} \\
&=& \partial _y \left (\frac{1}{|N_x|} \right) e_1 + b_x e_1 + b N_{xx},
\eeqas
so $\langle f_x, \partial _y e_1 \rangle = b^2 \hat \kappa_g$, and $a_y = (b^2+1) \hat \kappa_g$.
Hence the non-degeneracy condition is 
\[
\hat \kappa_g \neq 0.
\]

The speed of the curve $f(x,0)$, namely $b$, is arbitrary, and not related to the geometry of $\alpha$. Along $C$, we have
$N_y = b N_x$, and any choice of $b$ will generate a solution with $\alpha$ as a 
singular curve.  In terms of the curve $f(x,0)$, we have $f_x = b e_2$ along $C$, so $|b|$ is the speed of
this parameterization of $\gamma(x)=f(x,0)$. We also have $f_{xx} = b_x e_2 - \hat \kappa_g e_1$, so
the normal to $\gamma$ is $\bn = \pm e_1$ and the binormal along $\gamma$ is $\bb = \pm e_3$. 
Up to sign, the Frenet-Serret 
equations gives
\[
- \tau  e_1 = - \tau \bn = \frac{1}{|b|} \frac{\dd \bb}{\dd x} = \frac{1}{|b|}  (-N_x) = -\frac{1}{|b|}  e_1.
\]
Hence the torsion of the curve $f(x,0)$ is (up to a choice of sign)
\[
\tau = \frac{1}{b}.
\]

\subsection{Codimension 1 Singularities}
As in the previous section, if
the surface is a front, then either $f_x$ or $N_x$ is non-zero. In the case $f_x$ non-vanishing, as we
described above, the non-degeneracy condition is $\kappa \neq 0$.  In \cite{spherical}, it is shown that,
at a point where $\tau \neq 0$, then
$\kappa$ vanishes to first order if and only if the surface has a cuspidal beaks singularity at the point.
Here, the curve $C$ in the domain given by $y=0$ is contained in the singular set 
of $f$ but is not necessarily the whole singular set of $f$, as is the case at a non-transverse $A_3^-$.

To obtain both types of singularities (non-transverse  $A_3^-$ and $A_4$) from the same setup, we need to consider the case that
 $N_x$ is non-vanishing. Choose a frame $(e_1,e_2,e_3)=(r N_x, -rf_y, N)$ as above, with $r=1/|N_x|=1/|f_y|$ and $r(x,0)=1$,
$f_x=ae_1+be_2$ and $f_x \times f_y = \mu N$, with
\[
\mu= - a|f_y|, 
\]
and the null direction for $f$ is $\eta = \partial_x  + b \partial _y$.
By the recognition criteria in \cite{izumiyasaji,izsata}, a singularity of a front is diffeomorphic at $p$ to :
\begin{enumerate}
\item Cuspidal butterfly ($A_4$) if and only if $\hbox{rank}( \dd \mu) =1$, $\eta \mu(p) = \eta^2 \mu(p) = 0$ and \mbox{$\eta^3 \mu (p) \neq 0$};
\item Cuspidal beaks  (non-transverse $A_3^-$) if and only if $\hbox{rank}(\dd \mu) = 0$, $\det(\hbox{Hess} (\mu(p))) <0$ and $\eta^2 \mu(p)\neq 0$.
\end{enumerate} 

\begin{lemma}\label{lem:recog}
Let $f$ be as above. Then, at $p_0=(x_0,0)$:
\begin{enumerate}
\item   $f$ has a cuspidal butterfly singularity   if and only if $b(x_0,0) = b_x(x_0,0) = 0$ ,
   $b_{xx} (x_0,0) \neq 0$ and $\hat \kappa_g(p_0) \neq 0$.
\item $f$ has a cuspidal beaks singularity  if and only if  $\hat \kappa_g(p_0) = 0$, 
$\frac{\partial \hat \kappa_g}{\partial x}(p_0) \neq 0$ and $b(x_0,0) \neq 0$.
\end{enumerate}
\end{lemma}

\begin{proof}
The proof of the first item is given in \cite{spherical} (Theorem 4.8).
To prove the second item,
we have already shown that the degeneracy condition $\dd \mu = 0$ is equivalent to $\hat \kappa_g= 0$.
Here we have $\mu(x,0) = 0$ along $C$, so $\mu_{x}=\mu_{xx}=0$ along $C$. Noting that 
$|f_y|=1$ along $C$, and using the formula
$a_y = (b^2+1) \hat \kappa_g$ obtained earlier, we have
\[
\det (\hbox{Hess} (\mu(p))) = -(\mu_{xy})^2 = -\left(\frac{\partial}{\partial x}((b^2+1) \hat \kappa_g)\right)^2 
= -(b^2+1)^2\left(\frac{\partial \hat \kappa_g}{\partial x}\right)^2,
\]
at a degenerate point (i.e., a point where $\hat \kappa_g = 0$).  Thus the 
Hessian condition amounts to $\frac{\partial \hat \kappa_g}{\partial x} \neq 0$.
Finally, we have,
$
\eta^2 \mu = 2b (b^2+1) \frac{\partial \hat \kappa_g}{\partial x}$ along $C$, 
from which it follows that the three conditions for the cuspidal beaks are equivalent to  
$\hat \kappa_g(p_0) = 0$, $\frac{\partial \hat \kappa_g}{\partial x}(p_0)  \neq 0$ and $b(x_0,0) \neq 0$.
\end{proof}

\subsection{Codimension $\le 1$ singularities from geometric Cauchy data}
We show here how to obtain singularities of spherical frontals from Cauchy data along a regular 
curve in $\SSS^2$.
We  first re-state Theorem 4.8 from \cite{bjorling}, incorporating Lemma \ref{lem:recog}.

\begin{theorem}\label{gcptheorem}  {\rm(}\cite{bjorling}{\rm)}  
Let $I \subset \real$ be an open interval. Let $\hat \kappa_g: I \to \real$ and 
$b: I \to \real$ be a pair of real analytic functions.  Let $U \subset \C$ be a connected
open set containing the real interval $I \times \{0\}$ such that the functions $\hat \kappa_g$
and $b$ both extend holomorphically to $U$, and denote these holomorphic extensions 
by $\hat \kappa_g(z)$ and $\hat b(z)$.  Set
\[
\omega_{b,\hat\kappa_g} =\bbar 2 \hat \kappa_g(z)  i &  (-1-i b(z))\lambda^{-1} +(-1+i b(z))\lambda \\
   (1+i b(z))\lambda^{-1} +(1-i b(z))\lambda  &-  2 \hat \kappa_g(z) i  \ebar \dd z,
\]
Let $f: U \to \real^3$ and $N: U \to \SSS^2$ denote respectively the spherical frontal and 
its harmonic Gauss map produced from $\omega_{b,\hat\kappa_g}$ via the DPW method
of Section \ref{sec:DPW}.  Then $\kappa_g$ is
the geodesic curvature of the curve $N(x,0)$, and $b(x)=b(x,0)$ is the speed of the curve $f(x,0)$.
Moreover, the set $I \times \{0\}$ is contained in the singular set of $f$, and the singularity
at $(x_0, 0)$ is:
\begin{enumerate}
\item A cuspidal edge if and only if $b(x_0) \neq 0$ and $\hat \kappa_g (x_0) \neq 0$.
\item A swallowtail if and only if $b(x_0)=0$, $b^\prime(x_0) \neq 0$  and $\hat \kappa_g(x_0) \neq 0$.
\item A cuspidal butterfly if and only if $b(x_0) = b^\prime(x_0) =0$, $b^{\prime \prime}(x_0) \neq 0$,
and $\hat \kappa_g(x_0) \neq 0$.
\item A cone point if and only if $b(x,0) \equiv 0$ and $\hat \kappa_g(x_0) \neq 0$.
\item A cuspidal beaks if and only if 
$\hat \kappa_g(x_0) = 0$, 
$\frac{\partial \hat \kappa_g}{\partial x}(x_0) \neq 0$ and $b(x_0) \neq 0$.
\end{enumerate}
Conversely, all singularities of the types listed can locally be constructed in this way.
\end{theorem}

We can now prove the main results about the realization of the generic bifurcations of parallel surface by families of spherical surface at the cuspidal butterfly and cuspidal beaks singularities.

\begin{figure}[ht]
	\centering
	$
	\begin{array}{ccc}
	\includegraphics[height=27mm]{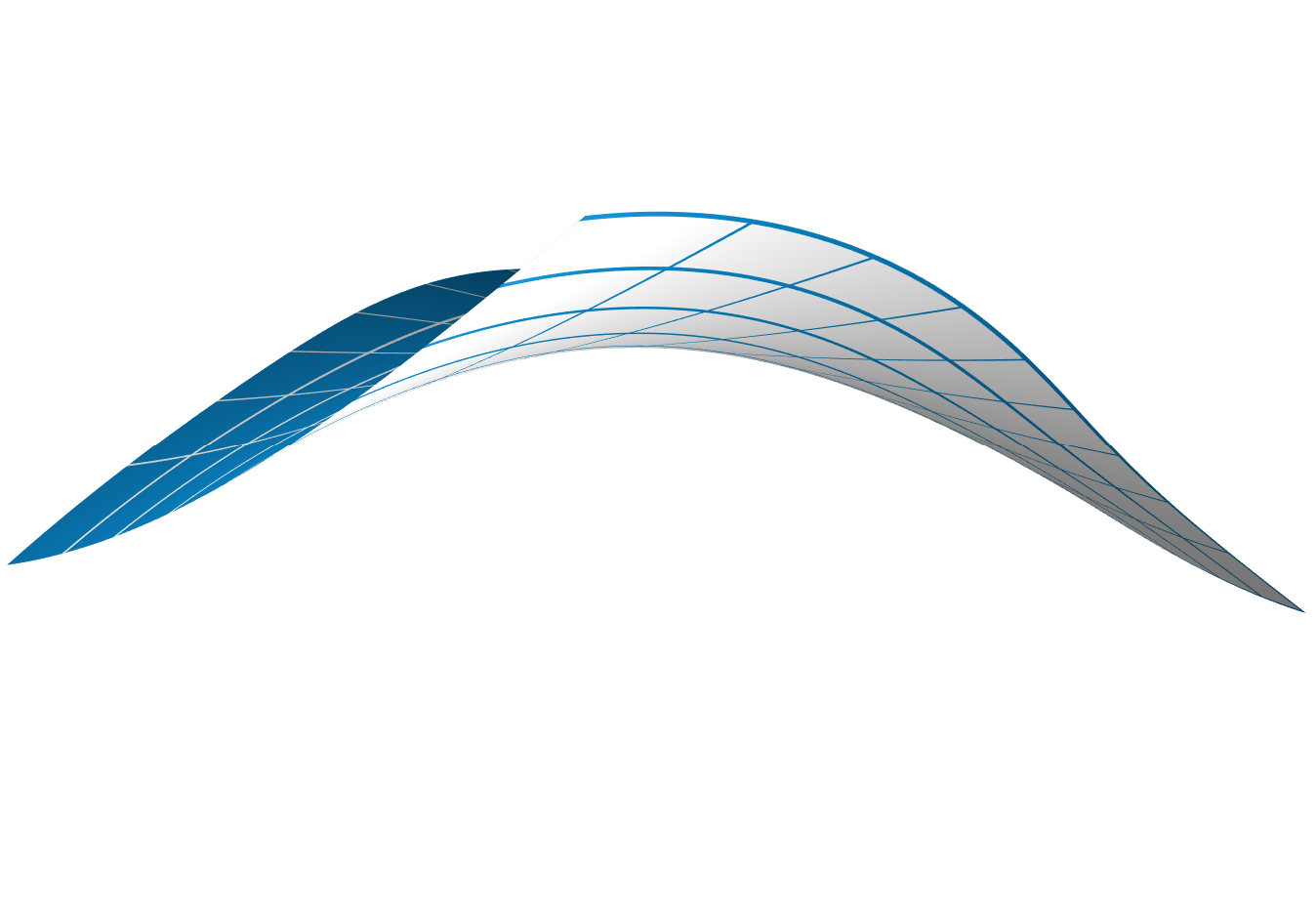}  \, & \, 
\includegraphics[height=27mm]{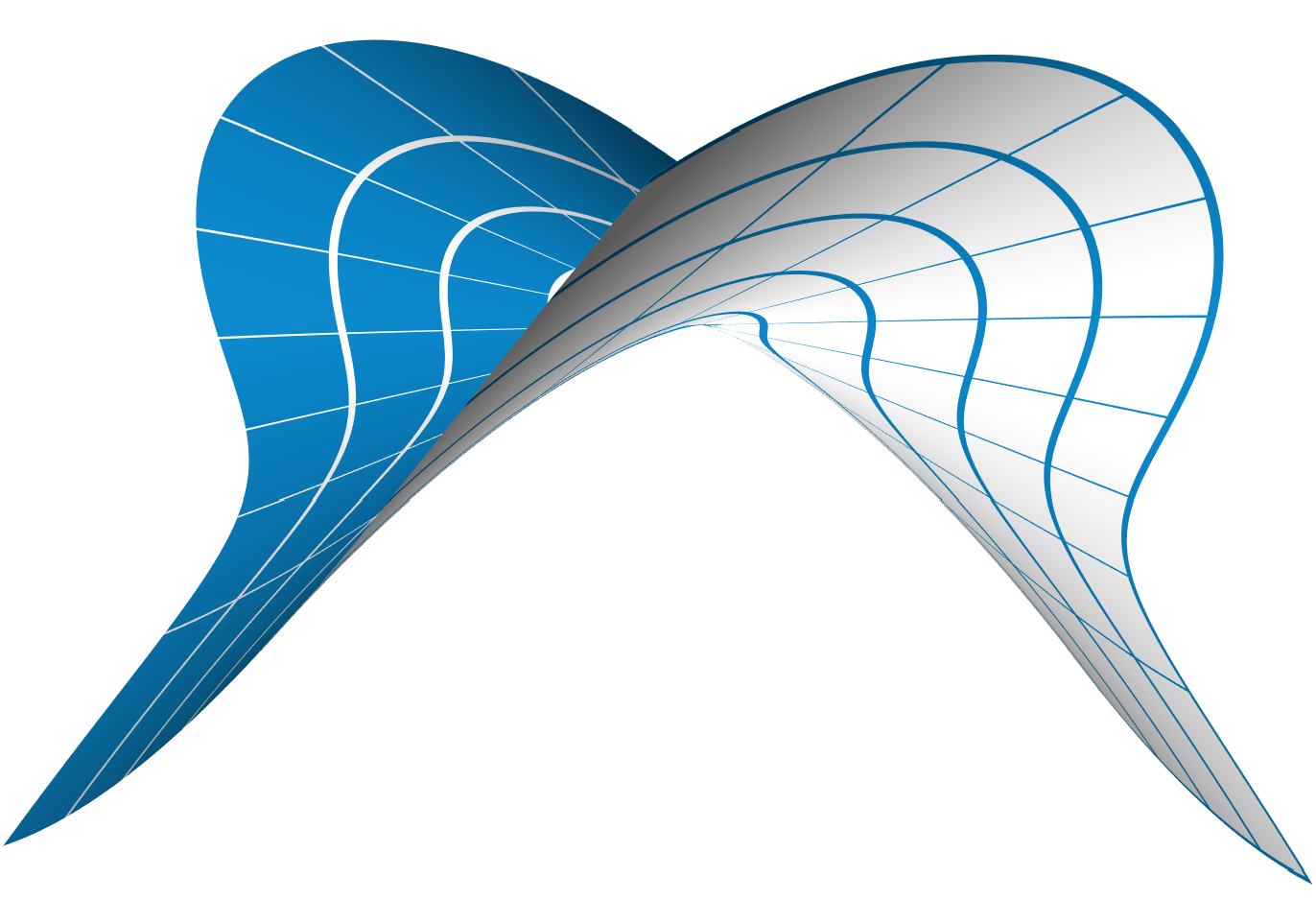}  \,& \, 
\includegraphics[height=27mm]{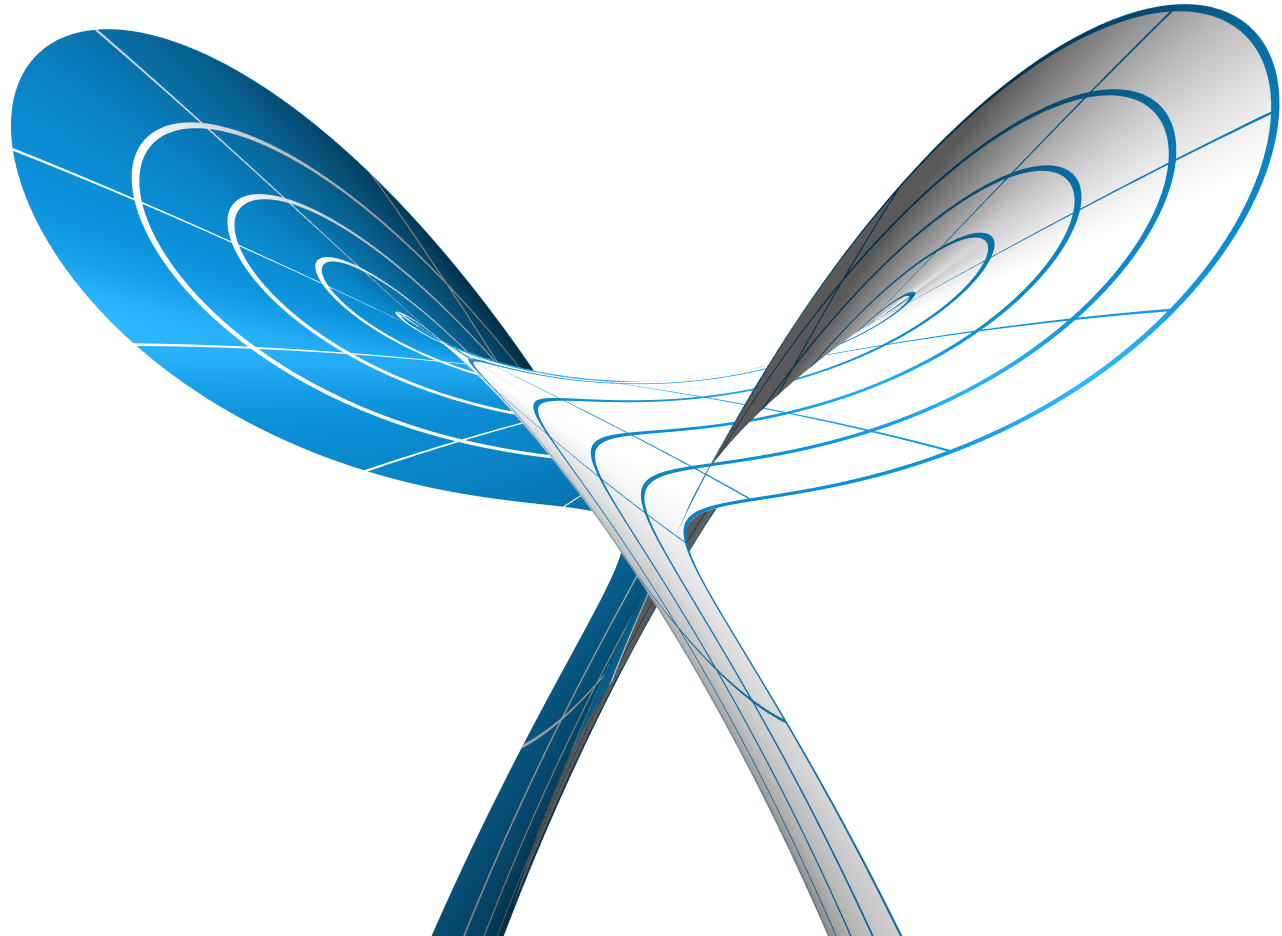}  
	\end{array}
	$
	\caption{Cuspidal butterfly bifurcation}
	\label{fig:butterfly_bif}
\end{figure}
\begin{theorem}\label{theo:evolA4}
	The evolution of wave fronts at an $A_4$ (butterfly) is realized in a
	1-parameter family of spherical surfaces $f_s(x,y) \equiv f(x,y,s)$ obtained from 
	Theorem \ref{gcptheorem} with the potentials:
	\[
\omega_{b_s,\hat\kappa_g}, \quad b_s(x)=s+x^2, \quad \hat \kappa_g(x)=1.
\]
\end{theorem}

\begin{proof}
We use the DPW-method in \S \ref{sec:DPW} 
to construct the 1-parameter family of spherical surfaces $f_s$ with potential as stated in the theorem. The surface $f_0$ has a butterfly singularity by Theorem \ref{gcptheorem}. 
The set $S_{A_3}$ associated to the family of distance squared functions on $f_s$ 
has a parameterization in the form $(x,0,q(x),-2x)$, so is a regular curve and its projection to the first two coordinates is also a regular curve. Then the result follows by Theorem \ref{theo:GemCritVersA4}.
\end{proof}

 The spherical surface $f_s$ in Theorem \ref{theo:evolA4} is computed numerically
 for three values of $s$ close to zero and shown in Figure \ref{fig:butterfly_bif}.
  
\begin{theorem}\label{theo:evolA3-}
	The evolution of wave fronts at a non-transverse $A_3^{-}$ 
	(cuspidal beaks) is realized in a
	1-parameter family of spherical surfaces $f_s$ obtained from 
	Theorem \ref{gcptheorem} with the potentials:
	\[
\omega_s = \omega_{b,\hat\kappa_{g}} + \alpha_s,
 \quad \quad b(x) = 1, \quad \hat \kappa_{g}(x) = x, \quad
 \alpha_s = \bbar 0 & s \lambda^{-1} \\ 0 & 0 \ebar \dd z.
\]
\end{theorem}

\begin{proof}
The DPW method described in Section \ref{sec:DPW} produces a family of surfaces $f_s$ from 
the potentials $\omega_s$, and the construction is real analytic in all parameters. 
Using the same notation as in Section \ref{sec:DPW}, we have  here
\[
b_0(z,s) = s-1-i, \quad \quad c_0(z,s) = 1+i.
\]
The  singular set, as noted before Lemma \ref{dpwlemma2}, is given at each $s$ by the equation
$\rho^2(x,y,s)|b_0(z,s)| = \rho^{-2}(x,y,s)|c_0(z,s)|$, i.e., by
\[
\rho^2(x,y,s)\sqrt{(s-1)^2+1} = \rho^{-2}(x,y,s)\sqrt{2}.
\]
Since we have a cuspidal beaks singularity at $s=0$, it follows from 
Theorem \ref{theo:recogCuspidalLipsBeaks} that the family $f_s$ realizes the
	evolution of the cuspidal beaks if and only if the set $h^{-1}(0)$, with
	\[
h(x,y,s) := \rho^4(x,y,s)-\frac{\sqrt{2}}{\sqrt{(s-1)^2+1} },
	\]
	is a regular surface in $\real^3$ in a neighbourhood of the point $(x,y,s)=(0,0,0)$, i.e.,
	$\dd h(0,0,0) \neq (0,0,0)$. But, in the DPW construction, where the integration point is
	$(x,y)=(0,0)$ for all $s$, we have $\rho(0,0,s)=1$ for all $s$, which implies that
	$\frac{\partial h}{\partial s} \neq 0$ at $(0,0,0)$, and the claim follows.
\end{proof}
Three solutions for $s$ close to zero are shown in Figure \ref{fig:beaks_bif}.
\begin{figure}[ht]
	\centering
	$
	\begin{array}{ccc}
	\includegraphics[height=30mm]{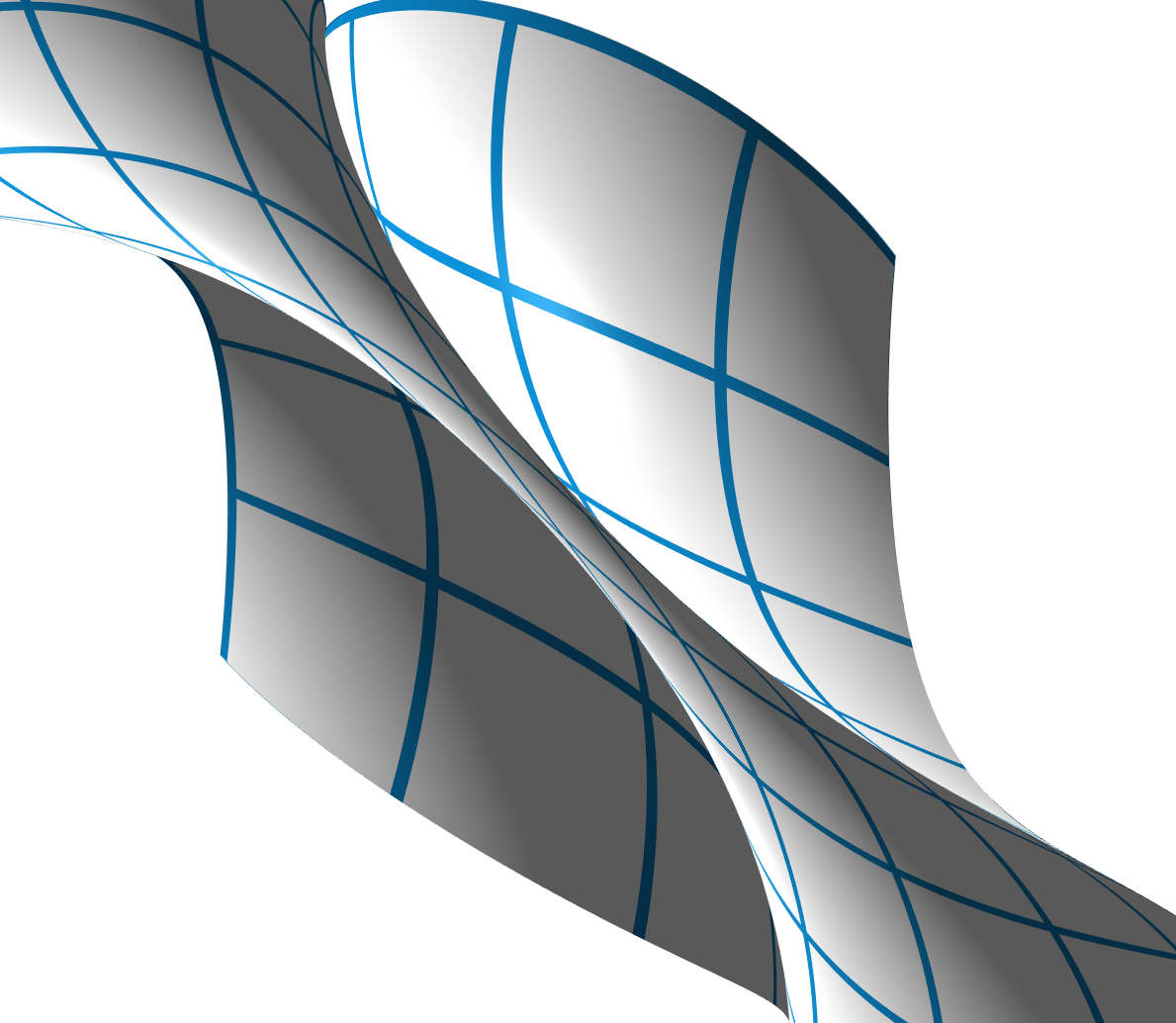}    \,\, \, & \, \, \,
\includegraphics[height=30mm]{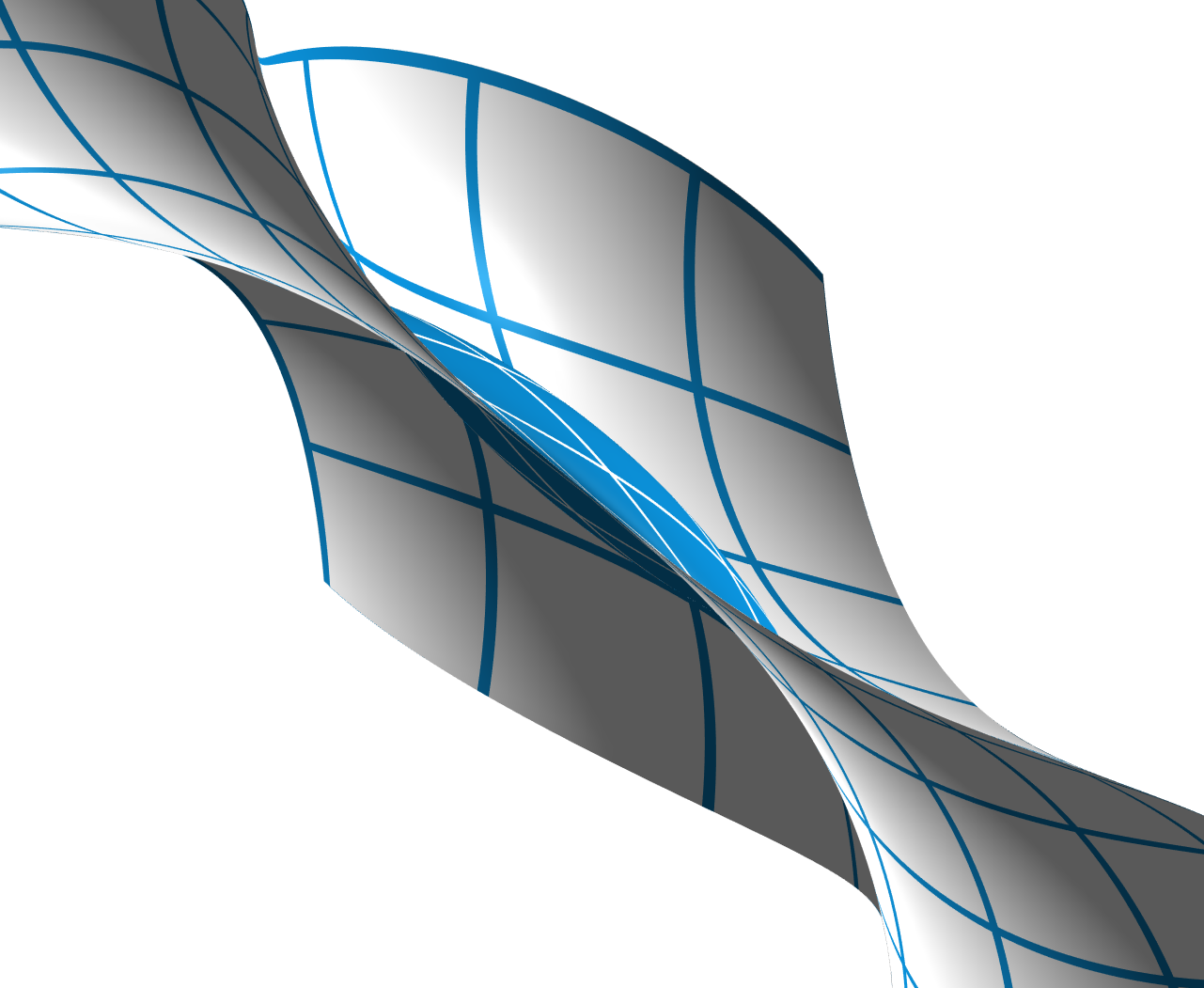}  \,\, \, & \, \, \,
	\includegraphics[height=30mm]{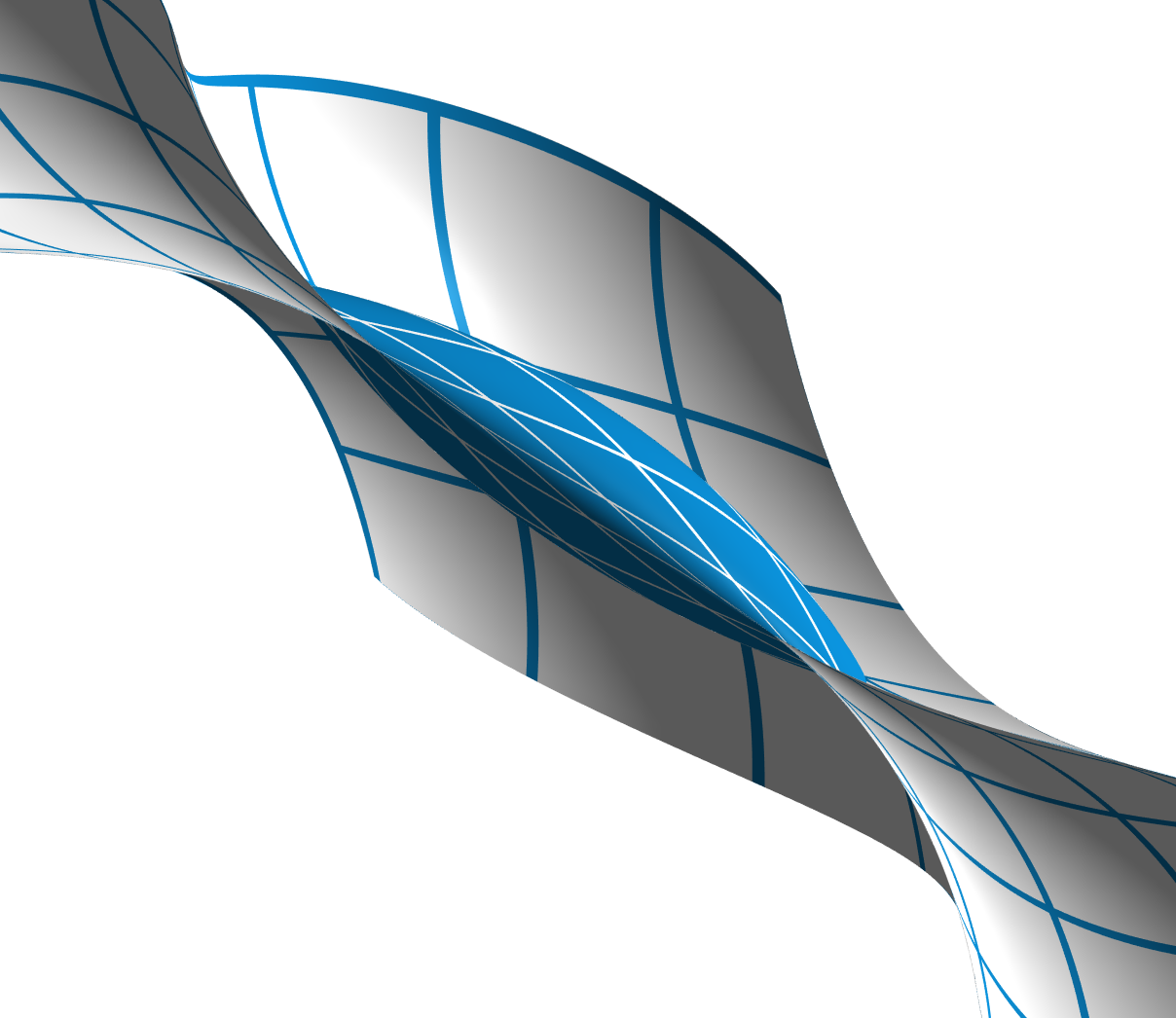}
	\end{array}
	$
	\caption{Cuspidal beaks bifurcations. }
	\label{fig:beaks_bif}
\end{figure}

\begin{remark} 
Theorem \ref{theo:PossibCodm1SingSS} lists the possible codimension 1 phenomena 
		of generic wave fronts that may occur in families of spherical surfaces. Theorems \ref{theo:evolA3-} and \ref{theo:evolA4} show that these do indeed occur. These bifurcations exhaust all possible bifurcations in generic 1-parameter families of spherical surfaces.
		In fact, given a $z_0 \in \C$, a harmonic map from an open  neighbourhood of $z_0$ into $\SSS^2$ is uniquely 
		determined by an arbitrary pair of holomorphic functions $(b_{-1}(z), c_{-1}(z))$, the so-called normalized potential
		at the point $z_0$ (see, e.g. \cite{DorPW}).
		An instability of a spherical surface $f$ is thus expressed in terms of the pair of  $(b_{-1}(z), c_{-1}(z))$.
		If further conditions 
		are imposed on $a_{-1}$ and $b_{-1}$, then these lead to codimension $\ge 2$ instabilities. 
\end{remark}

\section{Singularities of germs of harmonic maps from the plane to the plane}\label{sec:SingHarmMap} 
Wood \cite{wood1977} defined the following concepts for harmonic maps 
$N:D\to E$ between two Riemannian surfaces.
Denote by $\Sigma$ the singular set of $N$ and let $p\in \Sigma$.
The set $\Sigma$ is the zero set of the function 
$J=\det(\dd N)_p$. 
The point $p$ is called a {\it degenerate point} if 
$J$ vanishes identically in a neighbourhood of $p$.
It is called a {\it good} point
if $J$ is a regular function at $p$ (i.e., if $\Sigma$ is a 
regular curve at $p$). Then $(\dd N)_p$ has rank $1$. 
Wood classified the singularities of $N$ as follows: 

\begin{itemize}
	\item[(1)]$\rank(\dd N)_p=1$
	\begin{itemize}
		\item[(A)] {\it fold point} if $\nabla_{\sigma}N\ne 0$ at $p$ where 
		$\sigma$ is a tangent direction to $\Sigma$.
		
		\item[(B)] {\it collapse point} if $\nabla_{\sigma}N\equiv 0$ locally on 
		$\Sigma$.
		
		\item[(C)] {\it good singular point of higher order}
		if $\nabla_{\sigma}N= 0$ at $p$ but not identically zero 
		on $\Sigma$.
		
		\item[(D)] {\it $C^1$-meeting point of $s$ general folds} 
		if the singular set consists of pairwise transverse intersections of $s$ regular curves. 
	\end{itemize}
	
	\item[(2)] $\rank(\dd N)_p=0$: $p$ is called a {\it branch point}. 
\end{itemize}

\begin{theorem}{\rm (}\cite{wood1977}{\rm )} \label{theo:wood}
	Let $N: D\to E$ be a harmonic mapping between Riemannian surfaces.
	Then, if $p$ is a singular point for $N$, one of the following holds:
	\begin{enumerate}
	\item  $N$ is degenerate at $p$. Then: $N$ 
	must be degenerate on its whole domain and either (1A) $N$  is a constant mapping (so $\dd N\equiv 0$) or (IB) $\dd N$ 
	is nonzero except at isolated points of the domain.
	
\item $p$ is a good point: it may be (2A) fold point,  (2B) 
	collapse point or (2C) good singular point of higher order.
	
	\item $p$ is a $C^1$-meeting point 
	of an even number of 
	general folds. The general folds are arranged at equal 
	angles with respect to a local conformal coordinate system.
	
\item $p$ is a branch point.
\end{enumerate}
\end{theorem}

We are discussing here only harmonic maps between Euclidean planes
so we treat them as germs of mappings $N:(\mathbb R^2,0)\to (\mathbb R^2,0)$. 
Thus, $N=({\rm Re}(g_1),{\rm Re}(g_2))$  where 
$g_1,g_2$ are germs of holomorphic functions in $z=x+{\rm i}y$.
Such germs are of course special in 
the set $\mathcal{E}(2,2)$ of map-germs from the plane to the plane. 

Let $\mathcal{E}(2,1)$ denote the ring of germs of smooth functions $(\mathbb R^2,0)\to \mathbb R$, 
${\mathcal M}_2$ its unique maximal ideal and $\mathcal{E}(2,2)$ the $\mathcal{E}(2,1)$-module 
of smooth map-germs $(\mathbb R^2,0)\to \mathbb R^2$.
Consider the action of the group $\mathcal A$ of pairs of germs of smooth 
diffeomorphisms $(h,k)$ of the source and target on 
${\mathcal M}_2.\mathcal{E}(2,2)$ given by
$k\circ g\circ h^{-1}$, for $g\in {\mathcal M}_2.\mathcal{E}(2,2)$ 
(see, for example, \cite{arnoldetal,Martinet,Wall}). 
A germ $g$ is said to be finitely $\mathcal A$-determined if 
there exists an integer $k$ such that any map-germ with the same $k$-jet as $g$ 
is $\mathcal A$-equivalent to $g$. 
Let $\mathcal A_{k}$ be the subgroup of $\mathcal A$ whose elements
have the identity $k$-jets. 
The group $\mathcal A_k$ is a normal subgroup of $\mathcal
A$. Define
$\mathcal A^{(k)}=\mathcal A/\mathcal A_k$. The elements of $\mathcal A^{(k)}$ are the
$k$-jets of the elements of $\mathcal A$.
The action of $\mathcal A$ on ${\mathcal M}_2.{\mathcal E}(2,2)$
induces an action
of $\mathcal A^{(k)}$ on $J^k(2,2)$ as follows. For $j^kf\in J^k(2,2)$
and $j^kh\in \mathcal A^{(k)}$, $j^kh.j^kf=j^k(h.f).$

The tangent
space to the $\mathcal{A}$-orbit of $f$ at the germ $f$ is given by
 $$
L{\mathcal{A}}{\cdot}{f}={\mathcal M}_2.\{f_{x},f_{y}\}+f^*({\mathcal M}_2).\{e_1,e_2\},
$$
where $f_{x}$ and $f_y$ are the partial derivatives of $f$, $e_1,e_2$ denote
the standard basis vectors of ${\mathbb R}^2$ considered as elements
of ${\mathcal{E}}(2,2)$, and $f^*({\mathcal M}_2)$ is the pull-back of the maximal
ideal in ${\mathcal{E}}_2$. 
The extended tangent space to the $\mathcal{A}$-orbit of $f$ at the germ $f$
is given by
$$
L_e{\mathcal{A}}{\cdot}{f}={\mathcal E}_2.\{f_{x},f_{y}\}+f^*({\mathcal E}_2).\{e_1,e_2\}.
$$

We ask which finitely $\mathcal A$-determined singularities of map-germs in 
$\mathcal{E}(2,2)$ have a harmonic map-germ in their $\mathcal A$-orbit, that is, 
which singularities can be represented by a germ of a harmonic map.
We also ask whether an $\mathcal A_e$-versal deformation of the singularity
can be realized by families of harmonic maps. (This means that the initial 
harmonic map-germ can be deformed within the set of harmonic map-germs and 
the deformation is $\mathcal A_e$-versal.)

The most extensive classification of finitely $\mathcal A$-determined 
singularities of maps germs in $\mathcal{E}(2,2)$ of $\rank$ $1$ is carried out by Rieger in \cite{rieger} where he gave the following orbits of 
$\mathcal A$-codimension $\le 6$ (the parameters $\alpha$ and $\beta$ are moduli and take values in $\mathbb R$ with certain exceptional values removed, 
see \cite{rieger} for details):

$(x,y^2)$

$(x,xy+P_1(y))$, $P_1=y^3,\, y^4,\, y^5\pm y^7,\, y^5,\,  
y^6\pm y^8+\alpha y^9,\,  y^6+ y^9$ or $y^7\pm y^9+
\alpha y^{10}+\beta y^{11}$

$(x,y^3\pm x^ky)$, $k\ge 2$

$(x,xy^2+ P_2(y))$, $P_2=y^4+y^{2k+1} (k\ge 2),\, 
y^5+y^6,\,  y^5\pm y^9,\,  y^5$ or $y^6+y^7+\alpha y^9$

$(x,x^2y+P_3(x,y))$, $P_3=y^4\pm y^5,\,  y^4$ or $
xy^3+\alpha y^5+y^6+\beta y^7$

$(x,x^3y+\alpha x^2y^2+y^4 +x^3y^2)$

\medskip
The $\mathcal A$-simple map-germs of rank $0$  
are classified in \cite{riegerruas} and are as follows
$$
\begin{array}{l}
I_{2,2}^{l,m}=(x^2+y^{2l+1},y^2+x^{2m+1}),\, l\ge m\ge 1\\
I_{2,2}^{l}=(x^2-y^2+x^{2l+1},xy),\, l\ge 1.
\end{array}
$$

We answer the above two questions for the singularities in Rieger's list 
and for the $\mathcal A$-simple \mbox{rank $0$} map-germs.

\begin{proposition}\label{prop:goodsingpoint}
	{\rm (i)} The {\it fold} $(x,y^2)$ can be represented by the harmonic map 
	$(x,{\rm Re}(x+{\rm i}y)^2)$. It is $\mathcal A_e$-stable.
	
	{\rm (ii)} Any finitely $\mathcal A$-determined map-germ in $\mathcal{E}(2,2)$ 
	with a 2-jet $\mathcal A^{(2)}$-equivalent to $(x,xy)$ can be represented
	by a harmonic map. Furthermore, there is an $\mathcal A_e$-versal deformation of such germs 
	by harmonic maps.
\end{proposition}

\begin{proof} (i) is trivial. For (ii), since $\rank(\dd N)_0=1$, we can make holomorphic changes of coordinates and set
$g_1=z$. We can also take $j^1N=(x,0)$. 
	
Any $k$-$\mathcal A$-determined map-germ with 2-jet $(x,xy)$
is $\mathcal A$-equivalent to one in the form $(x,xy+P(y))$, where $P$ is a polynomial
in $y$ and $3\le {\rm degree}(P)\le k$ (see \cite{rieger}).
	
We have $j^kN=(x,\sum_{j=2}^k {\rm Re}({\bf a}_j(x+{\rm i}y)^j))$, with ${\bf a}_j=a_{j,1}+{\rm i}a_{j,2}$.
Changes of coordinates are carried out inductively on the jet level to reduce
$j^kN$ to the form $(x,xy+P(y))$. 
Monomials of the form $(0,x^p)$ are eliminated by changes of coordinate $(u,v)\mapsto
(u,v+cu^p)$ in the target, and monomials of the form
$(0,x^py^q)$, $q\ge 1$, are eliminated by changes of coordinate $(x,y)\mapsto
(x,y+cx^py^{q-1})$ in the source ($c$ an appropriate scalar). These 
inductive changes of coordinates reduce the $k$-jet of $N$ to the form $(x,xy+\sum_{j=3}^kc_jy^j)$
with $c_j={\rm Re}({\bf a}_j({\rm i})^j)+Q$, $Q$ a polynomial
in $a_{i,1},a_{i,2}$ with $i<j$. Clearly, the map $\mathbb C^{k-2}\to \mathbb R^{k-2}$
given by $({\bf a}_3,\ldots, {\bf a}_k)\mapsto (c_3,\ldots, c_k)$ is surjective and this proves 
the first part of the proposition.
	
For the second part, an $\mathcal A_e$-versal deformation of 
$g(x,y)=(x,xy+P(y))$ can be taken in the form 
$G(x,y,\lambda_1,\ldots,\lambda_l)=(x,xy+P(y)+\sum_{i=1}^l\lambda_iy^i)$, where the 
$\lambda_i$ are 
the unfolding parameters and $l$ is the $\mathcal A_e$-codimension of the singularity 
of $g$.
The same argument as above shows that $N+(0,\sum_{j=2}^k {\rm Re}
({\bf b}_j(x+{\rm i}y)^j))$ is
an $\mathcal A_e$-versal deformation of $N$, with ${\bf b}_j=b_{j,1}+{\rm i}b_{j,2}$ the unfolding 
parameters.
\end{proof}

Harmonic map-germs with finitely $\mathcal A$-determined singularities as in Proposition \ref{prop:goodsingpoint} are the {\it good singular points of higher order} in Wood's terminology.
Observe that the {\it collapse point} is a `singularity of a map-germ 
$\mathcal A$-equivalent to $(x,xy)$; this is not finitely $\mathcal A$-determined.

\begin{proposition}\label{prop:beaks}
	The only map-germ in the series $(x,y^3\pm x^ky)$, $k\ge 2$, that can be realized by 
	a harmonic map is the so-called beaks-singularity $(x,y^3-x^2y)$. 
	In that case, there exists an $\mathcal A_e$-versal deformation 
	of the singularity by a 1-parameter family of harmonic map-germs.
\end{proposition}

\begin{proof}
	Two $\mathcal A$-equivalent map-germs have diffeomorphic singular sets.
	The singular set of the map-germ $(x,y^3\pm x^ky)$ is given by  
	$3y^2\pm x^k$. For $k> 2$, this is either an isolated point, a single branch singular curve or two
	tangential smooth curves. All these cases are excluded for harmonic maps by Wood's Theorem \ref{theo:wood}. For $k=2$, the lips-singularity $(x,y^3+ x^2y)$
	is also excluded by Wood's theorem as the singular set is an isolated point.
	The beaks singularity can be realized by the harmonic map
	$(x,{\rm Re}({\rm i}(x+{\rm i}y)^3))=(x,y^3-3x^2y)$. The family 
	$(x,y^3-3x^2y+a_{1,2}y)$, which can be written as 
	$(x,{\rm Re}({\rm i}(x+{\rm i}y)^3-a_{1,2}{\rm i}(x+{\rm i}y)))$,
	is an $\mathcal A_e$-versal deformation and is given by harmonic map-germs.
\end{proof}

\begin{proposition}
	\label{prop:xy^2series}
	Any finitely $\mathcal A$-determined map-germ in $\mathcal{E}(2,2)$ 
	with a 3-jet $\mathcal A^{(3)}$-equivalent to $(x,xy^2)$ can be represented
	by a harmonic map. Furthermore, there are $\mathcal A_e$-versal deformations of such germs by harmonic maps.
\end{proposition}

\begin{proof}
	Any finitely $\mathcal A$-determined map-germ
	with a 3-jet $(x,xy^2)$ is $\mathcal A$-equivalent to 
	$(x,xy^2+P(y))$ for some polynomial function $P$ (see \cite{rieger}). 
	The proof then continues in a way similar to that of  Proposition \ref{prop:goodsingpoint}.
\end{proof}

\begin{proposition}\label{prop:x^2y}
{\rm (i)} 	Map-germs with a 3-jet $\mathcal A^{(3)}$-equivalent to $(x,x^2y)$ cannot be represented
	by a harmonic map.

{\rm (ii)} The map-germ $(x,x^3y+\alpha x^2y^2+y^4 +x^3y^2)$ cannot be represented
by a harmonic map.
\end{proposition}

\begin{proof}
Such germs have the form $(x,x^2y+P(x,y))$ with $j^3P\equiv 0$. Their singular set has equation
	$x^2+P_y(x,y)=0$, so has an $A_k$-singularity with $k\ge 2$. These are excluded 
	for harmonic maps by Wood's Theorem \ref{theo:wood}. Similarly, the singular set of the germ in (ii) has an odd number of branches (one or three) and this is excluded 
for harmonic maps by Wood's Theorem \ref{theo:wood}.
\end{proof}

Propositions \ref{prop:goodsingpoint}, \ref{prop:beaks}, \ref{prop:xy^2series}, \ref{prop:x^2y} together give a complete classification of finitely $\mathcal A$-determined singularities of harmonic 
map-germs with $j^3N$ not ${\mathcal A}$-equivalent to $(x,0)$.
When  $j^3N$ is ${\mathcal A}$-equivalent to $(x,0)$, we have the following.

\begin{proposition}\label{prop:3jet(x,0)}
Suppose that $j^kN$ is ${\mathcal A^{(k)}}$-equivalent to $(x,0)$, for $k\ge 5$. Then the singularity of 
	$N$ does not admit an $\mathcal A_e$-versal deformation by germs of 
	harmonic maps. 
\end{proposition}

\begin{proof}
Consider a finitely $\mathcal A$-determined 
germ $h(x,y)=(x,h_2(x,y))$ with $j^{k}h_2\equiv 0$. 
Then, since $j^{k-1}(\frac{\partial h_2}{\partial x})\equiv 0$ and 
	$j^{k-1}(\frac{\partial h_2}{\partial y})\equiv 0$, the space
	$\mathcal E(2,2)/L_e\mathcal A\cdot{}h$ contains the real vector 
	space $V$ of dimension $k(k-1)/2$ 
	generated by $(0,x^iy^j)$, with $0\le i\le k-1$, 
	$1\le j\le k-1$ and $i+j\le k-1$. 
	The $k$-jet of a deformation of a harmonic map $N\sim_{\mathcal A} h$ 
	by harmonic maps is of the form
	$\tilde N=N +(0,\sum_{j=1}^k {\rm Re}(b_{j,1}+{\rm i}{b_{j,2}})(x+{\rm i}y)^j)$. 
	For $\tilde N$ to be an $\mathcal A_e$-versal deformation,	 
	the polynomials $\frac{\partial \tilde N}{\partial b_{j,1}}$ and 
	$\frac{\partial \tilde N}{\partial b_{j,2}}$, $1\le j\le k-1$ must
	generate $V$. For this to be possible, we must have $2(k-1)\ge k(k-1)/2$, that is,
	$(k-1)(k-4)\le 0$, which holds if and only if $ 1\le k\le 4$.
\end{proof}

We turn now to the rank zero map-germs. 

\begin{proposition}\label{prop:Rank0}
	{\rm (i)} The singularity $I_{2,2}^{l,m}$ cannot be represented by
	a germ of a harmonic map.
	
	{\rm (ii)} The singularities $I_{2,2}^{l}$ can be represented by
	a germ of a harmonic map. Furthermore, there are $\mathcal A_e$-versal
	deformations  of these singularities by harmonic maps.
\end{proposition}

\begin{proof}
	(i) We have $j^2N=(a_{2,1}(x^2-y^2)-2a_{2,2}xy,
	b_{2,1}(x^2-y^2)-2b_{2,2}xy )$, so the 
	$\mathcal A$-orbits in the 2-jets space are 
	$(x^2-y^2,xy), (0,xy), (0,0)$ which proves the claim in (i) as the 2-jet of the 
 singularity $I_{2,2}^{l,m}$ is $\mathcal A$-equivalent to $(x^2,y^2)$.
	
	(ii) A simple calculation shows that the harmonic map 
	$N=(x^2-y^2+ {\rm Re}(x+{\rm i}y)^{2l+1}),xy)$
	is in the $\mathcal A$-orbit of $I_{2,2}^l$ 
	and admits the following $\mathcal A_e$-versal deformation
	$N+ (\sum_{j=1}^{l-1}{\rm Re}(\lambda_i(x+{\rm i}y)^{2i+1}),0)$ 
	by germs of harmonic maps.
\end{proof}

\begin{remark}\label{rem:AboutSingHarmMap}
	(1) Not all rank $0$ singularities of harmonic map-germs 
	(i.e., branch points) can be 
	$\mathcal A_e$-versally deformed by harmonic maps. 
	For instance, following 
	the calculation in the proof of Proposition \ref{prop:3jet(x,0)}(ii), 
	one can show that harmonic maps
	of the form 
	$(xy+h.o.t, {\rm Re}((a_{k,1}+{\rm i}a_{k,2})(x+{\rm i}y)^k)+h.o.t)$
	do not admit $\mathcal A_e$-versal deformations by harmonic maps 
	if $k\ge 7$.
	
	(2) The singular set cannot be an isolated point when $\rank(\dd N)_p=1$ 
	(Theorem \ref{theo:wood}). 
	But it can when $\rank(\dd N)_p=0$ (this is the case, for instance, for the 
	$I_{2,2}^{l}$-singularity).
\end{remark}

\vspace{0.5cm}
\noindent
{\bf Acknowledgement:} We are very grateful to the referee for valuable suggestions. 
Part of the work in this paper was carried out while the second author was a visiting professor at Northeastern University, Boston, Massachusetts, USA. He would like to thank Terry Gaffney and David Massey for their hospitality during his visit and FAPESP for financial support with 
the grant 2016/02701-4. He is partially supported by the grants 
FAPESP 2014/00304-2 and CNPq 302956/2015-8.


\begin{thebibliography}{10}
\providecommand{\url}[1]{{#1}}
\providecommand{\urlprefix}{URL }
\expandafter\ifx\csname urlstyle\endcsname\relax
  \providecommand{\doi}[1]{DOI~\discretionary{}{}{}#1}\else
  \providecommand{\doi}{DOI~\discretionary{}{}{}\begingroup
  \urlstyle{rm}\Url}\fi

\bibitem{ArnoldWavefront}
Arnold, V.: Wave front evolution and equivariant {M}orse lemma.
\newblock Commun. Pure and Appl. Math. \textbf{29}, 557--582 (1976)

\bibitem{Arnold79}
Arnold, V.: Indexes of singular points of 1-forms on manifolds with boundary,
  convolutions of invariants of groups generated by reflections, and singular
  projections of smooth surfaces.
\newblock Uspekhi Mat. Nauk \textbf{34}, 3--38 (1979)

\bibitem{arnoldetal}
Arnold, V., Gusein-Zade, S., Varchenko, A.: Singularities of differentiable
  maps {I}, \emph{Monographs in Mathematics}, vol.~82.
\newblock Birkh\"auser Boston, Inc. (1985)

\bibitem{arnold1990}
Arnold, V.I.: Singularities of caustics and wave fronts, \emph{Mathematics and
  its Applications (Soviet Series)}, vol.~62.
\newblock Kluwer Academic Publishers Group, Dordrecht (1990)

\bibitem{spherical}
Brander, D.: Spherical surfaces.
\newblock Exp. Math. \textbf{25}(3), 257--272 (2016).
\newblock DOI: 10.1080/10586458.2015.1077359

\bibitem{bjorling}
Brander, D., Dorfmeister, J.F.: The {B}j\"orling problem for non-minimal
  constant mean curvature surfaces.
\newblock Comm. Anal. Geom. \textbf{18}, 171--194 (2010)

\bibitem{bruceParallel}
Bruce, J.: Wavefronts and parallels in {E}uclidean space.
\newblock Math. Proc. Cambridge Philos. Soc. \textbf{93}, 323--333 (1983)

\bibitem{docarmo1}
do~Carmo, M.P.: Differential geometry of curves and surfaces.
\newblock Prentice-Hall (1976)

\bibitem{DorPW}
Dorfmeister, J., Pedit, F., Wu, H.: Weierstrass type representation of harmonic
  maps into symmetric spaces.
\newblock Comm. Anal. Geom. \textbf{6}, 633--668 (1998)

\bibitem{fukuihasegawa}
Fukui, T., Hasegawa, M.: Singularities of parallel surfaces.
\newblock Tohoku Math. J. (2) \textbf{64}, 387--408 (2012)

\bibitem{Gaffney}
Gaffney, T.: The structure of ${T}\mathcal{A}(f)$, classification and an
  application to differential geometry.
\newblock Proc. Sympos. Pure Math., \textbf{40}, 409--427 (1983)

\bibitem{Goryunov}
Goryunov, V.V.: Singularities of projections of complete intersections.
\newblock In: Current problems in mathematics, {V}ol. 22, Itogi Nauki i
  Tekhniki, pp. 167--206. Akad. Nauk SSSR, Vsesoyuz. Inst. Nauchn. i Tekhn.
  Inform., Moscow (1983)

\bibitem{ishimach}
Ishikawa, G., Machida, Y.: Singularities of improper affine spheres and
  surfaces of constant {G}aussian curvature.
\newblock Internat. J. Math. \textbf{17}, 269--293 (2006)

\bibitem{izumiyasaji}
Izumiya, S., Saji, K.: The mandala of {L}egendrian dualities for pseudo-spheres
  of {L}orentz-{M}inkowski space and ``flat'' spacelike surfaces.
\newblock J. Singul. \textbf{2}, 92--127 (2010)

\bibitem{izsata}
Izumiya, S., Saji, K., Takahashi, M.: Horospherical flat surfaces in hyperbolic
  3-space.
\newblock J. Math. Soc. Japan \textbf{62}, 789--849 (2010)

\bibitem{KabataR2R2}
Kabata, Y.: Recognition of plane-to-plane map-germs.
\newblock Topology Appl. \textbf{202}, 216--238 (2016)

\bibitem{KoenVanDo}
Koenderink, J., van Doorn, A.: The singularities of the visual mapping.
\newblock Biological Cybernetics \textbf{24}, 51--59 (1976)

\bibitem{krsuy}
Kokubu, M., Rossman, W., Saji, K., Umehara, M., Yamada, K.: Singularities of
  flat fronts in hyperbolic space.
\newblock Pacific J. Math. \textbf{221}, 303--351 (2005)

\bibitem{looijenga}
Looijenga, E.: Structural stability of smooth families of ${C}^{\infty}$ -
  functions.
\newblock PhD Thesis. Universiteit van Amsterdam (1974)

\bibitem{Martinet}
Martinet, J.: Singularities of smooth functions and maps, \emph{LMS Lecture
  Note Series}, vol.~58.
\newblock Cambridge University Press (1982)

\bibitem{Platonova}
Platonova, O.: Singularities of projections of smooth surfaces.
\newblock Russian Mathematical Surveys \textbf{39}(1), 177--178 (1984)

\bibitem{PreS}
Pressley, A., Segal, G.: Loop Groups.
\newblock Oxford Mathematical Monographs. Clarendon Press, Oxford (1986)

\bibitem{rieger}
Rieger, J.: Families of maps from the plane to the plane.
\newblock J. London Math. Soc. \textbf{36}, 351--369 (1987)

\bibitem{riegerruas}
Rieger, J., Ruas, M.: Classification of \mbox{$\mathcal A$}-simple germs from
  $k^n$ to $k^2$.
\newblock Compositio Math \textbf{79}, 99--108 (1991)

\bibitem{izusajitakashi}
S~Izumiya, K.S., Takahashi, M.: Horospherical flat surfaces in hyperbolic
  3-space.
\newblock J. Math. Soc. Japan \textbf{62}, 789--849 (2010)

\bibitem{SajiR2R2}
Saji, K.: Criteria for singularities of smooth maps from the plane into the
  plane and their applications.
\newblock Hiroshima Math. J. \textbf{40}, 229--239 (2010)

\bibitem{suy}
Saji, K., Umehara, M., Yamada, K.: The geometry of fronts.
\newblock Ann. of Math. (2) \textbf{169}, 491--529 (2009)

\bibitem{Wall}
Wall, C.: Finite determinacy of smooth map-germs.
\newblock Bull. London Math. Soc. \textbf{13}, 481--539 (1981)

\bibitem{Whitney}
Whitney, H.: On singularities of mappings of euclidean spaces. {I}. {M}appings
  of the plane into the plane.
\newblock Ann. of Math. (2) \textbf{62}, 374--410 (1955)

\bibitem{wood1977}
Wood, J.: Singularities of harmonic maps and applications of the
  {G}auss-{B}onnet formula.
\newblock Amer. J. Math. \textbf{99}, 1329--1344 (1977)

\end{thebibliography}
\end{document}